\documentclass{amsart}
\usepackage{commands}
\title{About a Minimal Model Program without flips}
\author{Alberto Chiecchio}
\date{}

\usepackage{overpic}

\begin{document}
\begin{abstract} We introduce a new vector space associated to a projective variety, the Weil N\'eron-Severi space, which we show is finitely generated and contains the usual N\'eron-Severi space as a subspace. We define the Nef cone of Weil divisor and the cone of Weil curves. We study these cones, and prove a new Cone theorem. We use this theorem to describe a Minimal Model Program without flips.
\end{abstract}

\maketitle

\tableofcontents

\section{Introduction}

In the regular Minimal Model Program, one naturally encounters non-$\Q$-Goren\-stein varieties -- i.e., varieties where no multiple of the canonical divisor $K_X$ is $\Q$-Cartier -- as targets of small contractions.  Since the entire Minimal Model Program is built on intersections with $K_X$, there is the need to introduce flips. However, this created major historical roadblocks. Until recently, it was not known that flips existed. Moreover, at the moment we only know that a specific sequence of flips terminates, and it is still conjectured that every sequence of flips terminates. One of the main motivations for defining a Minimal Model Program without flips is that such program would potentially avoid the issue of termination.

The first necessary tool to build such program is a good notion of positivity on non-$\Q$-factorial singularities, and this was done in \cite{wStefano2}. The second tool are a cone and a contraction theorem, and these are the main results of this paper. We also outline an algorithm for a Minimal Model Program without flips, whose outcome is the same as the regular MMP, i.e., models with $\Q$-factorial singularities and fibers of contractions that are log Fano. The key idea is that, instead of performing flips, we continue performing contractions and take a $\Q$-factorialization as last step -- i.e., a small, projective birational morphism whose source has $\Q$-factorial singularities. The next step is to prove termination of such program. I will discuss by the end of the introduction some of the issues. Moreover, the models we obtain are a priori different from the ones we obtain with the usual MMP, so it is also essential to understand what is the relation between these new models and the usual ones. Finally, we point out that several of the tools of this paper (e.g., the existence of $\Q$-factorializations) are a consequence of \cite{BCHM}. Thus this program is in no way alternative to the usual one, but it ought to be considered more as a parallel one.

\vspace{2ex}

To describe a program, we need to articulate the largest class of singularities where the program is expected to work, and show that we do not leave such class under contractions. Our program works on \emph{log terminal singularities}, in the sense of \cite{dFH}, over an algebraically closed field of characteristic $0$. This is essentially the same class as klt; indeed, $X$ has log terminal singularities if and only if there exists a divisor $\Delta$ such that $(X,\Delta)$ is a klt pair. In section \Sref{sect:lt}, we prove some generalizations of well known results. In some cases, the techniques used are the usual ones, but for some results, e.g. proposition \ref{prop:contr still lt}, we use the full power of positivity for reflexive sheaves.

In section \Sref{sect:NSW}, after introducing a notion of numerically trivial Weil divisors, we define the \emph{Weil N\'eron-Severi space $\NS(X)_W$} (definition \ref{df:NSW}) which is the vector space of Weil divisors, modulo the numerically trivial Weil divisors. One of the first things we prove is the finiteness of this space:

\begin{thm}[Theorem \ref{thm:NSWfg} and proposition \ref{prop:NS<NSW}] Let $X$ be a projective normal variety, defined over an algebraically closed field of characteristic $0$, with log terminal singularities. The space $\NS(X)_W$ is finite dimensional and naturally contains the regular N\'eron-Severi space as a subspace.
\end{thm}

Since we are working with Weil divisors, it is more natural to work with pushforwards than with pullbacks. In section \Sref{sect:NSW} we prove a rather technical result, upon which most of our work relies. Although fairly natural from the point of view of positivity, the complications lie in the fact that we are working with Weil divisors.

\begin{thm}[Theorem \ref{thm:nef<=>}]\label{thm:1.2} Let $X$ be a (normal) projective variety defined over an algebraically closed field of characteristic $0$ having log terminal singularities and let $f:Y\rightarrow X$ be a small, projective, birational morphism. A Weil $\R$-divisor $D$ on $Y$ is nef if and only if $D$ is $f$-nef and $f_*D$ is nef.
\end{thm}

Finally, we relate our notion of numerical triviality with the one of Fulton, \cite{fulton}.

In section \Sref{sect:cones} we define the Weil Nef cone $\Nef(X)_W$ (definition \ref{df:Nefcone}) and the cone of Weil curves $\NEbar(X)_W$, definition \ref{df:NEcone}. The former is naturally defined as the cone generated by the classes of nef Weil divisors in $\NS(X)_W$, while the latter is defined by duality. Thus, the elements of $\NEbar(X)_W$ are ``virtual'' curves. Since the usual N\'eron-Severi space is a subspace of $\NS(X)_W$, some faces of $\NEbar(X)_W$ might be defined by Cartier divisors. We call such faces \emph{rational} (definition \ref{df:rtlF}). In section \Sref{sect:cones} we also prove some basic results about the structure of this cones in relation with small birational morphisms, and we prove Kleiman's criterion for ampleness.

\begin{thm}[Theorem \ref{thm:kleimanWeil}] Let $X$ be a projective variety over an algebraically closed field of characteristic $0$ having log terminal singularities, and let $D$ be an $\R$-Weil divisor on $X$. Then $D$ is ample if and only if $\NEbar(X)_W\setminus\{0\}\subseteq D_{>0}$.
\end{thm}

In section \Sref{sect:conethm}, after proving a global generation theorem, theorem \ref{thm:globgen}, we finally prove the Cone theorem.

{\renewcommand{\labelenumi}{(\arabic{enumi})}
\begin{thm}[Cone theorem, Theorem \ref{thm:cone}]  Let $X$ be a normal projective variety over an algebraically closed field of characteristic $0$, with log terminal singularities.
\begin{enumerate}
\item There are (countably many) $C_j\in N_1(X)_W$ such that
$$
\NEbar(X)_W=\NEbar(X)_{W,K_X\geq0}+\sum\R_{\geq0}\cdot C_j,
$$
and they do not accumulate on the half-space $K_{X<0}$.

\item (Contraction theorem) Let $F\subset\NEbar(X)_W$ be a $K_X$-negative extremal face. There exists a diagram
\begin{equation}
\xymatrix{& \Xtilde\ar[dr]^{\varphitilde_F} \ar[dl]_f & \\ X \ar@{-->}[rr]_{\varphi_F} & & Y,}
\end{equation}
where 
\begin{enumerate}
\item $f$ is small, projective, birational and $\rho(\Xtilde)\leq\rho(X)+1$;
\item under the inclusion $\NEbar(X)_W\hookrightarrow\NEbar(\Xtilde)_W$, $F$ is a rational $K_{\Xtilde}$-negative extremal face and $\varphitilde_F$ is a contraction of the face $F$;
\item $\Xtilde$ has log terminal singularities and $\varphitilde_{F *}\regsh_{\Xtilde}=\regsh_Y$.
\end{enumerate}
\end{enumerate}
\end{thm}}

The trade-off of working with Weil divisors is that we only obtain rational maps, which are not always morphisms. This is not a big sacrifice, since also the regular MMP uses rational maps. However, since the cone $\NEbar(X)_W$ naturally surjects onto the usual cone of curves $\NEbar(X)_\R$, corollary \ref{cor:NEW>NE}, we obtain more maps. This is another reason why such program is worth exploring. By having more maps, we have a deeper understanding of the structure of the varieties. For example, from the point of view of Mori contractions, all Fano cones are alike (a cone has Picard number $1$, so the only possible Fano contraction is the map to a point). The Weil cone of curves has a richer structure, as example \ref{ex:cone2} shows.

In section \Sref{sect:program}, we propose a Minimal Model Program without flips.

\vspace{2ex}

What is still missing is precisely termination. We have a good understanding of how the cones change under small maps; the main obstacle is that we do not know in general what happens when the maps are not small. In some sense, the same tools that handle flipping contractions do not work so easily when the contraction is not a flipping one. In particular, we would need a result like theorem \ref{thm:1.2} above for general birational projective morphisms (with connected fibers). Future hope is that we can obtain termination by studying the interplay of the two spaces, $\NS(X)_\R$ and $\NS(X)_W$. We point out that, since the outputs that we obtain are a priori different from the ones obtained by the usual MMP, remark \ref{rk:models}, there is no immediate relation between the termination of our program and the usual termination.

\vspace{2ex}

Almost all the results in this paper work in the relative setting. To make the exposition more clear, we mostly worked in the non-relative setting, and we put all the relative versions in the last section, \Sref{sect:rel}.

\subsection*{Aknowledgments} This work started during my visit at Princeton University, and I am extremely indebted with Zs. Patakfalvi and M. Fulger for our conversations on which classes of singularities seemed more suitable for such project. I would like to thank M. Fulger for also suggesting lemma \ref{lm:mihai}. I would like to thank L. E. Miller and S. Kov\'acs for supporting me and having faith in this project, for listening to my frustrations when some proofs were escaping me, and for many suggestions on how to improve the exposition.

\section{Of Weil divisors (and other sorrows)}\label{sect:sorrows}

\subsection{Weil divisors and valuations}\label{subsect:Weil}
All the definitions and results in this subsection are of \cite{dFH}, with the exception of \ref{lm:pushforwardsheaves}, which is well-known.

\vspace{2ex}

Let $X$ be a normal variety. A \emph{divisorial valuation} on $X$ is a discrete valuation of the function field $k(X)$ of $X$ of the form $\nu=q\val_F$ where $q\in\R_{>0}$ and $F$ is a prime divisor over $X$, that is, a prime divisor on some normal variety birational to $X$. Let $\nu$ be a discrete valuation. If $\Idlsh$ is a \emph{coherent fractional ideal} of $X$ -- that is, a finitely generated sub-$\regsh_X$-module of the constant field $k(X)$ of rational functions on $X$ --, we set
$$
\nu(\Idlsh):=\min\{\nu(f)\,|\,\textrm{$f\in\Idlsh(U)$, $U\cap c_X(\nu)\neq\emptyset$}\}.
$$
\begin{df} To a given fractional ideal $\Idlsh$ we can associate a Weil divisor, called the \emph{divisorial part} of $\Idlsh$, as
$$
\div(\Idlsh):=\sum_{E\subset X}\val_E(\Idlsh)E,
$$
where the sum runs over all the prime divisors on $X$; equivalently, $\div(\Idlsh)$ is such that
$$
\regsh_X(-\div(\Idlsh))=\Idlsh\dual\dual.
$$
\end{df}
\begin{df} Let $\nu$ be a divisorial valuation on $X$. The \emph{$\natural$-valuation} or \emph{natural valuation} along $\nu$ of a divisor $F$ on $X$ is
$$
\nu^\natural(F):=\nu(\regsh_X(-F)).
$$
\end{df}
De Fernex and Hacon show, \cite[2.8]{dFH}, that, for every divisor $D$ on $X$ and every $m\in\Z_{>0}$, $m\nu^\natural(D)\geq\nu^\natural(mD)$ and
$$
\inf_{m\geq1}\frac{\nu^\natural(mD)}{m}=\liminf_{m\rightarrow\infty}\frac{\nu^\natural(mD)}{m}=\lim_{m\rightarrow\infty}\frac{\nu^\natural(m!D)}{m!}\in\R.
$$
\begin{df} Let $D$ be a divisor on $X$ and $\nu$ a divisorial valuation. The \emph{valuation along $\nu$} of $D$ is defined to be the above limit
$$
\nu^*(D):=\lim_{m\rightarrow\infty}\frac{\nu^\natural(m!D)}{m!}.
$$
\end{df}
\begin{df} Let $f:Y\rightarrow X$ be a proper birational morphism from a normal variety $Y$. For any divisor $D$ on $X$, the \emph{$\natural$-pullback} of $D$ along $f$ is
\begin{eqnarray}\label{eq:naturalpullbackdef}
f^\natural D:=\div(\regsh_X(-D)\cdot\regsh_Y);
\end{eqnarray}
equivalently, $f^\natural D$ is the divisor on $Y$ such that
\begin{eqnarray}\label{eq:naturalpullbackchar}
\regsh_Y(-f^\natural D)=(\regsh_X(-D)\cdot\regsh_Y)\dual\dual.
\end{eqnarray}
\end{df}
\begin{df} We define the \emph{pullback} of $D$ along $f$ as
$$
f^*D:=\sum_{E\subset Y}\val_E^*(D)E,
$$
where the sum runs over all prime divisors on $Y$. Equivalently,
$$
f^*D=\liminf_m\frac{f^\natural(mD)}{m}\textrm{ coefficient-wise}.
$$
\end{df}
\begin{prop}[\textrm{\cite[2.4, 2.10]{dFH}}] Let $\nu$ be a divisorial valuation on $X$ and let $f:Y\rightarrow X$ be a birational morphism from a normal variety $Y$. Let $D$ be any divisor and let $C$ be any $\R$-Cartier divisor, with $t\in\R_{>0}$ such that $tC$ is Cartier.
\begin{enumerate}
\item The definitions of $\nu^*(C)$ and $f^*(C)$ given above coincides with the usual ones. More precisely,
$$
\nu^*(C)=\frac{1}{t}\nu^*(tC)\quad and \quad f^*(C)=\frac{1}{t}f^*(tC).
$$
Moreover,
$$
\nu^*(tC)=\nu^\natural(tC)\quad and\quad f^\natural(tC)=f^*(tC).
$$
\item The pullback is almost linear, in the sense that
\begin{eqnarray}\label{eq:pullbacklinear}
f^\natural(D+tC)=f^\natural(D)+f^*(tC)\quad and\quad f^*(D+C)=f^*(D)+f^*(C).
\end{eqnarray}
\end{enumerate}
\end{prop}
\begin{lm}[\textrm{\cite[2.7]{dFH}}]\label{lm:(fg)*-g*f*} Let $f:Y\rightarrow X$ and $g:V\rightarrow Y$ be two birational morphisms of normal varieties, and let $D$ be a divisor on $X$. The divisor $(fg)^\natural(D)-g^\natural f^\natural(D)$ is effective and $g$-exceptional. Moreover, if $\regsh_X(-D)\cdot\regsh_Y$ is an invertible sheaf, $(fg)^\natural(D)=g^\natural f^\natural(D)$.
\end{lm}

We conclude with a very useful lemma regarding the behavior of divisorial sheaves along small birational morphisms.

\begin{lm}\label{lm:pushforwardsheaves} Let $f:Y\rightarrow X$ be small, proper, birational morphism of normal varieties, and let $D$ be a Weil divisor on $X$. Then $f_*\regsh_Y(D)=\regsh_X(f_*D)$.
\end{lm}

\begin{proof} Since $f$ is small, it is an isomorphism outside a codimension $2$ locus, so we have 
$$
\xymatrix{U \ar@{^(->}[r]^i \ar[d]_g^{\cong} & Y \ar[d]^f \\ V \ar@{^(->}[r]_j &X,}
$$
where $U$ and $V$ are open and $\codim(Y\setminus U)\geq2$, $\codim(X\setminus V)\geq2$. Let $D_U$ be the restriction of $D$ to $U$, and similar for $(f_*D)_V$. As Weil divisors, $g_*(D_U)=(f_*D)_V$. Then
\begin{eqnarray*}
f_*\regsh_Y(D)&=&f_*i_*\regsh_U(D_U)\cong j_*g_*\regsh_U(D_U)\cong j_*\regsh_V(g_*(D_U))=j_*\regsh_V((f_*D)_V)=\\
&=&\regsh_Y(f_*D),
\end{eqnarray*}
since both $\regsh_Y(D)$ and $\regsh_X(f_*D)$ are reflexive. Since $f_*\regsh_Y(D)$ is torsion-free, there is a natural inclusion $f_*\regsh_Y(D)\subseteq\regsh_X(f_*D)$; thus $f_*\regsh_Y(D)=\regsh_X(f_*D)$.
\end{proof}

\subsection{Positivity for Weil divisors}

In this subsection we will follow \cite{wStefano2}, and we will prove some technical results that we will need later.

\vspace{2ex}

Let us start with a key remark.

\begin{rk}\label{rk:globgen vs bpf} Positivity for Cartier divisors is sometimes defined in terms of global generation or basepoint-freeness. However, for linear systems associated to Weil divisors, basepoint-freeness and global generation are not equivalent. For example, consider the case of an affine cone and a non-$\Q$-Cartier divisor on it. Since the variety is affine, every coherent sheaf is globally generated; on the other hand, this divisor (and all its powers) must go through the vertex of the cone. In this case, every basepoint-free Weil divisor is necessarily $\Q$-Cartier. When thinking of how to define positivity for Weil divisors, presented with the choice between basepoint-freeness and global generation, the latter is more natural from the point of view of singularity theory. Basepoint-freeness seems to be a much stronger requirement. This choice led more naturally to cohomological statements, see \cite[\S4]{wStefano2}, but it might be less natural when applied in the context of the MMP, see the discussion at the beginning of section \Sref{sect:conethm}.
\end{rk}

Let us recall the definition of relative global generation.

\begin{df} Let $f:X\rightarrow U$ be a projective morphism of schemes, and let $\Fscr$ be a coherent sheaf on $X$. We say that $\Fscr$ is \emph{relatively globally generated} if the natural map $f^*f_*\Fscr\rightarrow\Fscr$ is surjective.
\end{df}

\begin{rk} This condition is local on the base, that is, it can be checked on open affine subschemes of $U$, where it reduces to the usual global generation.

As typical application of this observation, we can see that the product of two relatively globally generated coherent sheaves is still relatively globally generated.
\end{rk}

From now on we will fix a projective morphism of quasi-projective normal varieties $X\rightarrow U$.

\begin{df}[\textrm{cfr. \cite[2.3]{wStefano2}}] Let $f:X\rightarrow U$ be a projective morphism of quasi-projective normal varieties. A Weil $\Q$-divisor $D$ on $X$ is \emph{relatively asymptotically globally generated}, in short \emph{relatively agg} or \emph{$f$-agg}, if, for every positive $m$ sufficiently divisible, $\regsh_X(mD)$ is relatively globally generated.
\end{df}

Generalizing \cite[4.1]{stefano2}, in \cite{wStefano2} we gave the following definition.

\begin{df}[\textrm{\cite[4.1]{stefano2}, \cite[2.4]{wStefano2}}] Let $f:X\rightarrow U$ be a projective morphism of quasi-projective normal varieties over a noetherian ring $k$. A Weil $\Q$-divisor $D$ on $X$ is \emph{relatively nef}, or \emph{$f$-nef}, if for every relatively ample $\Q$-Cartier divisor $A$, $\regsh_X(D+A)$ is $f$-agg.

If $U=\Spec k$ (so that $X$ is projective) we will simply say that $D$ is \emph{nef}.
\end{df}

\begin{rk}\label{rk:failadditivity} It is not clear whether this notion is additive in general. Let $D_1$ and $D_2$ be  two nef Weil divisors such that for all $m\geq1$, $\regsh_X(mD_1)\otimes\regsh_X(mD_2)\neq\regsh_X(mD_1+mD_2)$. For any ample Cartier divisor $A$ and for any $m\gg0$, the sheaves $\regsh_X(mA+mD_1)$ and $\regsh_X(mA+mD_2)$ are globally generated; thus so is their tensor product $\regsh_X(mA+mD_1)\otimes\regsh_X(mA+mD_2)$. However, by assumption we have $\regsh_X(mA+mD_1)\otimes\regsh_X(mA+mD_2)\neq\regsh_X(m2A+m(D_1+D_2))$, and the sheaf on the right is the reflexive hull of the sheaf on the left. Thus, it is not immediately possible to deduce the global generation of $\regsh_X(m2A+m(D_1+D_2))$.
\end{rk}

The following lemma was mentioned in \cite{wStefano2}, but the proof was left to the reader. We will write it here for completion. 

\begin{rk} Following \ref{nota:lt}, in the results in this section, `$X$ with log terminal singularities' can be read as `$(X,\Delta)$ is a klt pair, for some boundary $\Delta$', see \ref{thm:dFH7.2}.
\end{rk}

\begin{lm}\label{lm:nefnessadditive} Let $f:X\rightarrow Y$ be a projective morphism of normal varieties, and let $X$ be defined over an algebraically closed field of characteristic $0$ and with log terminal singularities; (relative) nefness is additive.
\end{lm}
\begin{proof} The statement is local on the base, so we can assume $Y$ is affine. Consequently, it is enough to show that nefness is additive on $X$. By \cite[92]{kollarexercises}, all the algebras $\Rscr(X,D)$ are finitely generated (for $\Q$-divisors). Let $D_1$ and $D_2$ be nef Weil divisors and let $A$ be any ample $\Q$-Cartier divisor. By definition, $A/2+D_1$ and $A/2+D_2$ are agg. Let $m_1$  be  an integer such that, for every positive $m$ divisible by $m_1$ (and by $2$), $\regsh_X(m(A/2+D_1))$ is globally generated. Let $m_2$ be similarly defined for $D_2$ and let $m_0=\mathrm{lcm}(m_1,m_2)$. Then, for each positive $m$ divisible by $m_0$, $\regsh_X(m(A/2+D_1))$ and $\regsh_X(m(A/2+D_2))$ are globally generated. Hence, so is their tensor product
$$
\regsh_X(m(A/2+D_1))\otimes\regsh_X(m(A/2+D_2)).
$$
This sheaf naturally surjects onto
$$
\regsh_X(m(A/2+D_1))\cdot\regsh_X(m(A/2+D_2)),
$$
which is therefore also globally generated. Since all the algebras of local sections are finitely generated, for $m$ sufficiently divisible
$$
\regsh_X(m(A/2+D_1))\cdot\regsh_X(m(A/2+D_2))\cong\regsh_X(m(A+D_1+D_2)).
$$
Thus $D_1+D_2$ is nef.
\end{proof}

\begin{df}[\textrm{\cite[2.14]{wStefano2}}] Let $f:X\rightarrow U$ be a projective morphism of quasi-projective normal varieties over a noetherian ring $k$. A Weil $\Q$-divisor $D$ on $X$ is called \emph{relatively almost ample}, or \emph{almost $f$-ample}, if, for every relatively ample $\Q$-Cartier divisor $A$ there exists a $b>0$ such that $bD-A$ is $f$-nef. It is called \emph{relatively ample}, or \emph{$f$-ample}, if, in addition, $\Rscr(X,D)$ is finitely generated.

If $U=\Spec k$, in the two cases above we will simply say \emph{almost ample} or \emph{ample}.
\end{df}

\begin{rk} If $X$ is projective over a field of characteristic $0$ with log terminal singularities, for all $\Q$-divisors, the algebras $\Rscr(X,D)$ are finitely generated, \cite[92]{kollarexercises}, thus the notions of almost ample and ample coincide. In general, almost ample is not equivalent to ample, as \cite[2.20]{wStefano2} shows.
\end{rk}

As in standard setting (\cite[1.4.10]{lazarsfeld}), we have the following result (the following version is a slight strengthening of \cite[3.2]{wStefano2}, but the proof is essentially the same).

\begin{lm}\label{lm:amplecone=interiornefcone} Let $f:X\rightarrow U$ be a projective morphism of normal quasi-projective varieties over a noetherian ring $k$. Let $D$ be a divisor on $X$ and let $H$ be an $f$-ample divisor. Let us assume that either $D$ or $H$ is $\Q$-Cartier. Then $D$ is $f$-nef if and only if, for all sufficiently small $\e\in\Q$, $0<\e\ll1$, $D+\e H$ is $f$-ample.

If $k$ is an algebraically closed field of characteristic $0$ and $X$ has log terminal singularities, then assumption of $D$ or $H$ being $\Q$-Cartier is not needed.
\end{lm}

\begin{proof} The first part of the statement is \cite[3.2]{wStefano2}. So let us assume that $k$ is an algebraically closed field of characteristic $0$ and that $X$ has log terminal singularities.

Let us assume that $D+\e H$ is $f$-ample for all $\e\in\Q$, $0<\e\ll1$. Let $A$ be an $f$-ample $\Q$-Cartier $\Q$-divisor. For rational sufficiently small positive $\e$, $A-\e H$ is $f$-ample. Since all the algebras of local sections are finitely generated, \cite[92]{kollarexercises}, for all positive $m$ sufficiently divisible,
$$
\regsh_X(m(D+A))=\regsh_X(m(D+\e H)+(A-\e H))\cong\regsh_X(m(D+\e H))\cdot\regsh_X(m(A-\e H)).
$$
Since both $D+\e H$ and $A-\e H$ are $f$-ample, they are $f$-agg, and thus so is $D+A$. By definition, $D$ is $f$-nef.

Conversely, let $D$ be $f$-nef. By substituting $H$ with $\e H$, we reduce to prove that $D+H$ is $f$-ample. Let $A$ be $\Q$-Cartier and $f$-ample. For $1\gg\delta>0$, $H-\delta A$ is $f$-ample; moreover $D+\delta A$ is $f$-agg. Again, using \cite[92]{kollarexercises}, it is not hard to see that $D+H=(D+\delta A)+(H-\delta A)$ is $f$-ample.
\end{proof}

Assuming that the Weil divisor can be restricted to the generic fiber of $X\rightarrow U$, it is also possible to define relative bigness and pseudo-effectivity. This will be an implicit assumption in the next two definitions. We point out that this is true in all the cases we are interested in.

\begin{df}[\textrm{\cite[2.16]{wStefano2}}] Let $f:X\rightarrow U$ be a projective morphism of quasi-projective normal varieties over a noetherian ring $k$. A Weil $\Q$-divisor $D$ is called \emph{relatively big}, or \emph{f-big}, if there exist an $f$-ample $\Q$-Cartier divisor $A$ and an effective Weil $\Q$-divisor $E$ such that $D\sim_{f,\Q} A+E$.

If $U=\Spec k$, we will simply say that $D$ is \emph{big}.
\end{df}

\begin{df}[\textrm{\cite[2.17]{wStefano2}}] Let $f:X\rightarrow U$ be a projective morphism of quasi-projective normal varieties over a noetherian ring $k$. A Weil $\Q$-divisor $D$ is called \emph{relatively pseudo-effective}, or \emph{$f$-pseff}, if for every $f$-ample $\Q$-Cartier divisor $A$, $D+A$ is $f$-big.

If $U=\Spec k$, we will simply say that $D$ is \emph{pseudo-effective}.
\end{df}

The main characterization of these notions uses \emph{$\Q$-Cartierizations}.

\begin{lm}[\textrm{\cite[6.2]{komo}}]\label{lm:themostusefullemmaintheworld} Let $X$ be a normal algebraic variety and let $B$ be a Weil divisor. The following are equivalent:
\begin{enumerate}
\item $\Rscr(X,B)$ is a finitely generated $\regsh_X$-algebra;
\item there exists a small, projective birational morphism $f:Y\rightarrow X$ such that $Y$ is normal, and $\bar{B}:=f^{-1}_*B$ is $\Q$-Cartier and $f$-ample.
\end{enumerate}
Moreover, $f:Y\rightarrow X$ is unique with these properties, namely $Y\cong\Proj_X\Rscr(X,D)$, and, for all $m\geq0$, $f_*\regsh_Y(m\bar{B})=\regsh_X(mB)$.
We call such a map the \emph{$\Q$-Cartierization} of $D$.
\end{lm}

\noindent Notice that, if $D$ on $X$ is $\Q$-Cartier, the $\Q$-Cartierization of $D$ is $\id:\Xtilde=X\rightarrow X$.

\begin{thm}[\textrm{\cite[3.3, 3.6, 3.10, 3.12]{wStefano2}}]\label{thm:QCart positivity} Let $X\rightarrow U$ be a projective morphism of normal projective varieties over an algebraically closed field $k$. Let $D$ be a Weil divisor on $X$ and let $g: Y\rightarrow X$ be the $\Q$-Cartierization of $D$ (over $U$), with $\Dtilde:=g^{-1}_*D$; then $D$ is nef/ample/big/pseff over $U$ if and only if so is $\Dtilde$.
\end{thm}

\section{Log terminal singularities}\label{sect:lt}

In this section we recall some definitions and results due to \cite{dFH} (we will use the notation of \cite{wStefano}). Subsequently, we prove generalizations of well-known results. 

\vspace{2ex}

If $f:Y\rightarrow X$ is a proper birational morphism (of normal varieties) and if we choose a canonical divisor $K_X$ on $X$, we will always assume that the canonical divisor $K_Y$ on $Y$ be chosen such that $f_*K_Y=K_X$ (as Weil divisors).

\begin{df}[\textrm{cfr. \cite[3.1]{wStefano}}] Let $f:Y\rightarrow X$ be a proper birational map of normal varieties. The \emph{$m$-limiting relative canonical $\Q$-divisors} are
\begin{eqnarray*}
&&K^-_{m,Y/X}:=K_Y-\frac{1}{m}f^\natural(mK_X),\quad K^-_{Y/X}:=K_Y-f^*(K_X)\\
&&K^+_{m,Y/X}:=K_Y+\frac{1}{m}f^\natural(-mK_X),\quad K^+_{Y/X}:=K_Y+f^*(-K_X).
\end{eqnarray*}
\end{df}
As shown by \cite{dFH} (and as from the definitions), for all $m,q\geq1$,
\begin{eqnarray}\label{eq:K-<=K^+}
K^-_{m,Y/X}\leq K^-_{qm,Y/X}\leq K^-_{Y/X}\leq K^+_{Y/X}\leq K^+_{mq,Y/X}\leq K^+_{m,Y/X}
\end{eqnarray}
and
\begin{eqnarray}\label{K-(Y/X)=limsupK-(m,Y/X)}
K^-_{Y/X}=\limsup K^-_{m,Y/X}, \quad K^+_{Y/X}=\liminf K^+_{m,Y/X}
\end{eqnarray}
(coefficient-wise).

\begin{df}[\textrm{cfr. \cite[4.1]{wStefano}}] Let $Y\rightarrow X$ be a proper birational morphism with $Y$ normal, and let $F$ be a prime divisor on $Y$. For each integer $m\geq1$, the \emph{$m$-limiting discrepancy} of $F$ with respect to $X$ is
$$
a_m(F,X):=\ord_F(K^-_{m,Y/X}).
$$
\end{df}

\begin{df}[\textrm{\cite[7.1]{dFH}}] A variety $X$ is said to have \emph{log terminal singularities} if there exists an integer $m_0$ such that $a_m(F,X)>-1$ for every prime divisor $F$ over $X$ and $m=m_0$ (and hence for any positive multiple $m$ of $m_0$).
\end{df}
\begin{thm}[\textrm{\cite[7.2]{dFH}}]\label{thm:dFH7.2} Let $X$ be a normal variety defined over an algebraically closed field of characteristic $0$. Then $X$ has log terminal singularities if and only if there exists a boundary $\Delta$ on $X$ such that $(X,\Delta)$ is klt.
\end{thm}

\begin{nota}\label{nota:lt} In this paper will always use the terminology `log terminal singularities' in the above sense, so we do not assume the variety to be $\Q$-Gorenstein.
\end{nota}

\begin{lm}\label{lm:QCart still lt} Let $X$ be a normal variety over an algebraically closed field of characteristic $0$ having log terminal singularities, and let $f:Y\rightarrow X$ be a small, proper, birational morphism. Then $Y$ still has log terminal singularities.
\end{lm}

\begin{proof} Let $h:W\rightarrow X$ be any birational morphism factoring through $f$, and let $g:W\rightarrow Y$ so that $h=g\circ f$. Since $X$ has log terminal singularities, there exists $m_0$ such that, for every $m$ divisible by $m_0$, $K_W-\frac{1}{m}h^\natural(mK_X)>-1$, meaning, all the valuations along prime exceptional divisors are strictly bigger than $-1$. Let $E_m^g=\frac{1}{m}\big(h^\natural(mK_X)-g^\natural f^\natural(mK_X)\big)$, which is effective and $g$-exceptional, by \ref{lm:(fg)*-g*f*}. Thus,
\begin{eqnarray*}
-1&<&K_W-\frac{1}{m}h^\natural(mK_X)=K_W-\frac{1}{m}g^\natural f^\natural(mK_X)-E^g_m\leq\\
&\leq&K_W-\frac{1}{m}g^\natural f^\natural(mK_X).
\end{eqnarray*}
Since $f$ is small, $f^\natural(mK_X)=mK_Y$, which implies that 
$$
-1<K_W-\frac{1}{m}g^\natural f^\natural(mK_X)=K_W-\frac{1}{m}g^\natural(mK_Y),
$$
which proves that $Y$ has log terminal singularities.
\end{proof}

\begin{rk} The technique above shows a well-known behavior: if $f:\Xtilde\rightarrow X$ is a small morphism, the singularities (meaning terminal, canonical, lt, lc) of $\Xtilde$ are no worse than the singularities of $X$.
\end{rk}

As a partial converse of \ref{lm:QCart still lt}, we have the following partial generalization of \cite[7.4]{kolhighI} or \cite[3.5]{fujino_kawamata}.

\begin{prop}\label{prop:contr still lt} Let $X$ be a normal variety over an algebraically closed field of characteristic $0$, having log terminal singularities, $f:X\rightarrow Y$ be a projective, birational morphism, with $-K_X$ $f$-agg, then $Y$ still has log terminal singularities.
\end{prop}

\begin{proof} Let $m$ be any integer. Since $f_*(-mK_X)=-mK_Y$, $(f_*\regsh_X(-mK_X))\dual\dual=\regsh_Y(-mK_Y)$. Since $X$ and $Y$ are integral, $f$ is dominant and $\regsh_X(-mK_X)$ is torsion-free, $f_*\regsh_X(-mK_X)$ is also torsion-free, \cite[8.4.5]{ega}. Thus we obtain an inclusion
$$
f_*\regsh_X(-mK_X)\hookrightarrow(f_*\regsh_X(-mK_X))\dual\dual=\regsh_Y(-mK_Y).
$$
If we pullback the above inclusion, we obtain a map
\begin{equation}\label{eq:pullbackpushforward}
f^*f_*\regsh_X(-mK_X)\rightarrow f^*\regsh_Y(-mK_Y).
\end{equation}
Since $f$ is generically an isomorphism, the kernel of \eqref{eq:pullbackpushforward} is a torsion sheaf. Moreover, $\regsh_X\cdot f_*\regsh_X(-mK_X)$ is the image of $f^*f_*\regsh_X(-mK_X)$ in $k(X)\cong k(Y)$, so it is the quotient of $f^*f_*\regsh_X(-mK_X)$ by its torsion (see \cite[II.7.12.2]{har}). Similarly, $\regsh_X\cdot\regsh_Y(-mK_Y)$ is the quotient of $f^*\regsh_Y(-mK_Y)$ by its torsion. So we obtain an inclusion
\begin{equation}\label{eq:inj}
\regsh_X\cdot f_*\regsh_X(-mK_X)\hookrightarrow \regsh_X\cdot\regsh_Y(-mK_Y).
\end{equation}

Since $-K_X$ is relatively asymptotically globally generated, by definition for each positive $m$ sufficiently divisible we have a surjection
\begin{equation}\label{eq:f-agg,surj}
f^*f_*\regsh_X(-mK_X)\twoheadrightarrow\regsh_X(-mK_X).
\end{equation}
Since $\regsh_X(-mK_X)$ is torsion-free and $\regsh_X\cdot f_*\regsh_X(-mK_X)$ is the quotient of the sheaf $f^*f_*\regsh_X(-mK_X)$ by its torsion (as above), the surjection in \eqref{eq:f-agg,surj} descends to a surjection
\begin{equation}\label{eq:surj}
\regsh_X\cdot f_*\regsh_X(-mK_X)\twoheadrightarrow\regsh_X(-mK_X).
\end{equation}

Let $g:W\rightarrow X$ and $h=f\circ g: W\rightarrow Y$ be log resolutions of $(X,\regsh_X(-mK_X))$ and $(Y,\regsh_Y(-mK_Y))$ respectively. The integer $m$ can be chosen sufficiently divisible so that we have \eqref{eq:surj} and $K_W-\frac{1}{m}g^\natural(mK_X)$ has coefficients strictly larger than $-1$ (the latter is by definition of log terminal singularities). Now the discussion proceeds as above. Pulling back \eqref{eq:inj}, we obtain the map
$$
g^*\big(\regsh_X\cdot f_*\regsh_X(-mK_X)\big)\rightarrow g^*\big(\regsh_X\cdot\regsh_Y(-mK_Y)\big),
$$
whose kernel is torsion, and thus descends to an injection
\begin{equation}\label{eq:injW}
\regsh_W\cdot f_*\regsh_X(-mK_X)\hookrightarrow \regsh_W\cdot\regsh_Y(-mK_Y)
\end{equation}
(notice that, for any $\Fscr\subseteq k(Y)$, $\regsh_W\cdot\regsh_X\cdot\Fscr=\regsh_W\cdot\Fscr$). Pulling back \eqref{eq:surj} we obtain the surjection
$$
g^*\big(\regsh_X\cdot f_*\regsh_X(-mK_X)\big)\twoheadrightarrow g^*\regsh_X(-mK_X),
$$
which descends to a surjection
\begin{equation}\label{eq:surjW}
\regsh_W\cdot f_*\regsh_X(-mK_X)\twoheadrightarrow\regsh_W\cdot\regsh_X(-mK_X).
\end{equation}

Considering \eqref{eq:injW} and \eqref{eq:surjW} together, we obtain the following diagram
\begin{equation}\label{eq:diagram}
\xymatrix{ & \regsh_W\cdot\regsh_Y(-mK_Y)\\
\regsh_W\cdot f_*\regsh_X(-mK_X) \ar@{^(->}[ur] \ar@{->>}[dr] & \\
 &  \regsh_W\cdot\regsh_X(-mK_X).}
\end{equation}
Recall that $(\regsh_W\cdot\regsh_Y(-mK_Y))\dual\dual=\regsh_W(-h^\natural(mK_Y))$ and $(\regsh_W\cdot\regsh_X(-mK_X))\dual\dual=\regsh_W(-g^\natural(mK_X))$. Hence, if we look at the divisorial part of \eqref{eq:diagram}, we obtain that $-g^\natural(mK_X)\leq -h^\natural(mK_Y)$. This implies that
$$
K_W-\frac{1}{m}h^\natural(mK_Y)\geq K_W-\frac{1}{m}g^\natural(mK_X).
$$
The divisor on the right hand side has coefficient strictly larger than $-1$ since $X$ has log terminal singularities and $m$ was chosen sufficiently divisible. This implies that $K_W-\frac{1}{m}h^\natural(mK_Y)$ has coefficient strictly larger than $-1$ as well, and since $h$ is a log resolution $(Y,\regsh_Y(-mK_Y))$, $Y$ has log terminal singularities, \cite[\S 7]{dFH}.
\end{proof}

As an immediate corollary of these techniques we can deduce the existence of non-$\Q$-factorial log flips for log terminal singularities. This result is an immediate consequence of the work of \cite{BCHM}, and beyond the scope of this paper, so we will not linger on it.

\begin{cor} Let $f : X \rightarrow Y$ be a flipping contraction with respect to a klt pair $(X,\Delta)$ (over an algebraically closed field of characteristic $0$). Then the flip of $f$ exists, and $Y$ has log terminal singularities.
\end{cor}

\begin{proof} Indeed $X$ has log terminal singularities. Since $f$ is a flipping contraction, $K_X+\Delta$ is $f$-anti-ample. With very similar techniques to the ones of the previous proof, $Y$ has log terminal singularities, which by \cite[92]{kollarexercises} implies that $\Rscr(Y,K_Y+f_*\Delta)$ is finitely generated.
\end{proof}

The above statement \ref{prop:contr still lt} assumes the map to be birational, but only asks for $-K_X$ to be $f$-agg. If we assume that $-K_X$ is $f$-ample, we can remove the assumption that $f$ is birational, see \ref{prop:contr still lt,2}. It is reasonable to expect that the result holds more generally for any projective map with connected fibers $f:X\rightarrow Y$, $X$ with log terminal singularities and $-K_X$ $f$-agg. Before we show this, we need the following lemma.

\begin{lm}\label{lm:logFano} Let $f:X\rightarrow Y$ be a projective morphism between normal quasi-projective varieties over an algebraically closed field of characteristic $0$. If $X$ has log terminal singularities and $-K_X$ is $f$-ample, we can find a boundary $\Delta$ on $X$ such that $(X,\Delta)$ is a klt pair and $-(K_X+\Delta)$ is $f$-ample.
\end{lm}

\begin{proof} The proof is the same as \cite[5.7]{wStefano2}, but in a relative setting. Let $g:W=\Proj_X\Rscr(X,-K_X)\rightarrow X$ be the $\Q$-Cartierization of $-K_X$. Since $g$ is projective and small, $W$ has log terminal singularities, \ref{lm:QCart still lt}; moreover $W$ is $\Q$-Gorenstein. 

Let $h:=f\circ g$. By \ref{thm:QCart positivity}, $-K_W$ is $\Q$-Cartier and $h$-ample. Let  $A$ be any $f$-ample Cartier divisor on $X$ . Since $-K_W$ is $h$-ample, there exists $b>1$ such that $-bK_W-g^*A$ is $h$-globally generated. Let $M$ be the general element in the system $M\in |\regsh_W (-bK_W-g^*A)|$. Notice that $-bK_W-bM \sim g^*A$, which is trivial on the fibers. Let $\Delta:=g_*(M/b)$. Since $K_W+M/b\sim_{\Q}-g^*A$, $K_X+\Delta$ is $\Q$-Cartier, and $K_X+\Delta\sim_{\Q}-A$. Finally, since $W$ has $\Q$-Gorenstein log terminal singularities, and $M$ was general, $(Y,M/b)$ is a klt pair. Notice that $g_*(M/b)=\Delta$ and $K_W+M/b=g^*(K_X+\Delta)$ (by construction). Thus $(X,\Delta)$ is a klt pair \cite[2.30]{komo}.
\end{proof}

\begin{prop}\label{prop:contr still lt,2} Let $X$ be a normal variety over an algebraically closed field of characteristic $0$ having log terminal singularities, and let $f:X\rightarrow Y$ be a projective morphism with connected fibers, with $-K_X$ $f$-ample. Then $Y$ still has log terminal singularities.
\end{prop}

\begin{proof} The proof follows almost verbatim the one in \cite[3.5]{fujino_kawamata}. Since $-K_X$ is $f$-ample and $X$ has log terminal singularities, we can find $\Delta$ such that $(X,\Delta)$ is klt and $-K_X-\Delta$ is $f$-ample, \ref{lm:logFano}.

Let $H$ be an ample $\Q$-Cartier $\Q$-divisor on $Y$ such that $H':=-(K_X+\Delta)+f^*H$ is an ample $\Q$-Cartier divisor. Let $n\Delta$ be an integral divisor. Since $H'$ is ample, for $m$ sufficiently divisible $n\Delta+mH'$ is globally generated. Let $M\in|\regsh_X(n\Delta+mH')|$ be a general member. Since $M$ is general and $H'\sim_{\Q}\frac{1}{m}(M-n\Delta)$, $(X,\Delta+\e/m(M-n\Delta))$ is a klt pair for $0<\e\ll1$, \cite[2.43]{komo}. Let $m'$ be sufficiently divisible, so that $m'H'$ is very ample, and let $M'\in |\regsh_X(m'H')|$ be a general member. As before, $(X,\Delta+\e/m(M-n\Delta)+(1-\e)/m' M')$ is a klt pair. However,
$$
K_X+\Delta+\e/m(M-n\Delta)+(1-\e)/m' M'\sim_{\Q,f}0.
$$
By \cite[0.2]{fujino_kawamata} we have the desired result.
\end{proof}

\section{The Weil N\'e{}ron-Severi space}\label{sect:NSW}

\begin{df} Let $f:X\rightarrow U$ be a projective morphism of quasi-projective normal varieties. We say that a $\Q$-Weil divisor $D$ on $X$ is \emph{$f$-numerically trivial} if it is $f$-nef and $f$-anti-nef (i.e., $-D$ is $f$-nef).

If $X$ is projective, we simply say that $D$ is \emph{numerically trivial} if it is nef and anti-nef.
\end{df}

\begin{rk} We do not need nefness to be additive to give the above definition. However, when it is, the set of numerically trivial $\Q$-Weil divisors is a subgroup of $\Div(X)_\Q$.
\end{rk}

\begin{df}\label{df:NSW} Let $X\rightarrow U$ be a projective morphism of quasi projective varieties over an algebraically closed field of characteristic $0$, with $X$ having log terminal singularities. We define $\NS(X/U)_{W,\Q}$ to be the space of $\Q$-Weil divisors modulo the $f$-numerically trivial ones, and by $\NS(X/U)_W:=\NS(X/U)_{W,\Q}\otimes_\Q\R$.

If $U=\Spec k$, we will use the notation $\NS(X)_W:=\NS(X/\Spec k)_W$.
\end{df}

\begin{prop}\label{prop:NS<NSW} Let $X$ be a normal projective variety over an algebraically closed field of characteristic $0$ having log terminal singularities . There is a natural injection $\NS(X)_\R\hookrightarrow\NS(X)_W$.
\end{prop}

\begin{proof} Let $\pi:\Div(X)_\R\rightarrow\NS(X)_\R$ and $\pi_C:\CDiv(X)_\R\rightarrow\NS(X)_W$ be the two quotient maps. We have the commutative diagram
$$
\xymatrix{ 0 \ar[r] & \ker\pi \ar[r] & \Div(X)_\R \ar[r] & \NS(X)_W \ar[r] & 0 \\
0 \ar[r] & \ker\pi_C \ar@{^(->}[u] \ar[r] & \CDiv(X)_\R \ar@{^(->}[u] \ar[r] & \NS(X)_\R  \ar@{^(-->}[u] \ar[r] & 0.}
$$
Notice that $\ker\pi$ are all the $\R$-Weil divisors which are numerically trivial, while $\ker\pi_C$ are all the $\R$-Cartier divisors which are numerically trivial. Since $\ker\pi_C=\ker\pi\cap\CDiv(X)_\R$, the existence and the injectivity of the map $\NS(X)_\R\hookrightarrow\NS(X)_W$ follows form standard diagram chasing.
\end{proof}

It is well known that $\NS(X)_\R$ is a finitely generated vector space. The standard proof, see \cite{kle}, uses the fact that, if $f:X\rightarrow Y$ is a proper, birational morphism, then the pullback induces an injection $f^*:\NS(Y)_\R\hookrightarrow\NS(X)_\R$. Indeed, clearly the pullback of a nef $\Q$-Cartier divisor is nef. Morever, if $C$ is a curve on $Y$, there is a curve $C'$ on $X$ such that $f(C')=C$, \cite[lemma 4.1]{kle}, and $D.C=f^*D.C'$ for every $\Q$-Cartier divisor $D$ on $Y$.

Using this, we can reduce to the smooth case (using a resolution), where the result is well-known. Since the pullback of Weil divisors is not numerically trivial on the fibers in general (see the discussion in \cite{ltII}), a proper, birational morphism $f:X\rightarrow Y$ will not necessarily induce a map $\NS(Y)_W\rightarrow\NS(X)_W$ and, even when such map exists, it will not necessarily be injective. Morevover, we do not define nefness for Weil divisors with intersection with curves, so the technique of \cite{kle} does not work in this case. We will prove that $\NS(X)_W$ is finitely generated in \ref{thm:NSWfg}, but we need some preliminary results.

The following lemma is a non-$\Q$-Cartier version of \cite[Proposition 4.1]{kle}. The proof is a simplified version of the one of \cite[4.3]{stefano2}.

\begin{lm}\label{lm:kleiman4.1} Let $f:Y\rightarrow X$ be a small projective, birational morphism of normal projective varieties over an algebraically closed field. For an ample $\R$-divisor $A$ on $Y$, $f_*A$ is nef.
\end{lm}

\begin{proof} It is enough to assume that $A$ is a $\Q$-divisor. Moreover, it is enough to assume that $A$ is $\Q$-Cartier. Indeed, let us assume the result for $\Q$-Cartier ample divisors. Let $g:Z\rightarrow Y$ be the $\Q$-Cartierization of $A$ (which exists since $\Rscr(X,A)$ is finitely generated), and let $A_Z:=g^{-1}_*A$. By \ref{thm:QCart positivity}, $A_Z$ is ample $\Q$-Cartier, and $(fg)_*A_Z=f_*A$.

So let us assume that $A$ is ample $\Q$-Cartier on $Y$, and let $B$ be an ample $\Q$-Cartier divisor on $X$. We need to show that, for $m$ sufficiently divisible, $\regsh_X(m(f_*A+B))$ is globally generated. Since $f^*B$ is nef, $A+f^*B$ is still ample, and $f_*\regsh_Y(m(A+f^*B))=\regsh_X(m(f_*A+B))$, by \ref{lm:pushforwardsheaves}. Thus is it enough to show that, if $A$ on $Y$ is ample $\Q$-Cartier, $f_*A$ is agg.

Let $L$ be any very ample Cartier divisor on $X$. By Serre's vanishing, for $m$ sufficiently divisible
$$
H^i(Y,\regsh_Y(mA-i f^*L))=0,\quad i>0.
$$
Similarly, by the relative Serre's vanishing, for $m$ sufficiently divisible
$$
R^jf_*\regsh_Y(mA-if^*L)=0,\quad i,j>0.
$$
Thus, the Leray's spectral sequence converges immediately and
$$
H^i(X,f_*\regsh_Y(mA-if^*L))\cong H^i(Y,\regsh_Y(mA-i f^*L))=0,\quad i>0
$$
for $m$ sufficiently divisible. Since $f_*\regsh_Y(mA-if^*L)\cong\regsh_X(mf_*A-iL)$, by Castelnuovo-Mumford regularity $\regsh_X(mf_*A)$ is globally generated for $m$ sufficiently divisible.
\end{proof}

\begin{qt}\label{qt:pushforwardnef} Is the previous statement true without the assumption of $f$ being small? In this case, the same reasoning as above gives that $f_*\regsh_X(mA)$ is globally generated for $m$ sufficiently divisible, and that $H^1(Y,f_*\regsh_Y(mA))=0$. Since the surjective image of torsion free sheaves (between integral schemes) is still torsion free, we have an exact sequence
$$
0\rightarrow f_*\regsh_X(mA)\rightarrow\regsh_Y(mf_*A)\rightarrow \Qscr\rightarrow 0,
$$
where $\Qscr$ is a torsion sheaf. If $\dim\Supp\Qscr=0$, then $\Qscr$ is globally generated, and thus so is $\regsh_Y(mf_*A)$. It is unclear if we have the desired global generation in general.
\end{qt}

\begin{cor}\label{cor:pushforwardnef} Let $f:Y\rightarrow X$ be a small, projective, birational morphism of normal projective varieties  over an algebraically closed field of characteristic $0$, with $X$ having log terminal singularities. For a nef $\R$-divisor $A$ on $Y$. Then $f_*A$ is nef.
\end{cor}

\begin{proof} Notice that $Y$ still has log terminal singularities, \ref{lm:QCart still lt}. We can write $A$ as a limit of ample $\R$-divisors $A_m$. Then $f_*A=\lim f_*A_m$. Since nefness is a closed property by \ref{lm:amplecone=interiornefcone}, $f_*A$ is nef.
\end{proof}

The next result is in line \cite[5.4]{BdFFU}, where it is shown that the space of Weil divisors modulo the numerically Cartier is finitely generated. The two results are independent, however, as there is no relation between being numerically Cartier or being numerically trivial. Indeed, if the singularities are klt over an algebraically closed field of characteristic $0$, all numerically Cartier divisors are automatically $\Q$-Cartier, \cite[5.8]{BdFFU}.

\begin{thm}\label{thm:NSWfg} Let $g:X\rightarrow U$ be a projective morphism of normal quasi-projective varieties over an algebraically closed field of characteristic $0$, with $X$ having log terminal singularities. The space $\NS(X/U)_W$ is finitely dimensional.
\end{thm}

\begin{proof} We will prove the statement when $U=\Spec k$. The proof in the relative setting follows with the obvious modifications. Notice that, when $X$ itself is projective, $\NS(X/U)_W$ is a quotient of $\NS(X)_W$, so the proof below is enough for our study.

Let $f:Y\rightarrow X$ be a $\Q$-factorialization of $X$, that is, $f$ is a small, projective, birational morphism with $Y$ normal and $\Q$-factorial. This map exists since the singularities are log terminal, see \cite[108]{kollarexercises}. Then, $f^*$ is an isomorphism $f^*:\Div(X)_\R\rightarrow\Div(Y)_\R\cong\CDiv(Y)_\R$. Let $N_Y$ be the subspace of numerically trivial $\R$-Weil divisors on $Y$, and similarly $N_X$ the subspace of numerically trivial $\R$-Weil divisors on $X$. Then $\NS(X)_W=\Div(X)_\R/N_X$ and $\NS(Y)_W=\NS(Y)_\R=\Div(Y)/N_Y$. Let $N=N_X\cap(f^*)^{-1} N_Y$. The pullback does not induce a morphism $\NS(X)_W\rightarrow\NS(Y)_W$; however, it induces a morphism $\Div(X)_\R/N\rightarrow\NS(Y)_\R$, and $\Div(X)_\R/N$ naturally surjects onto $\NS(X)_W$. Since $\NS(Y)_\R$ is finitely generated, it is enough to prove that the map $\Div(X)_\R/N\rightarrow\NS(Y)_\R$ is injective, yielding the following diagram:
\begin{equation}\label{eq:f*surj}
\xymatrix{\Div(X)_\R \ar[r]^{f^*}_{\cong} \ar[d] & \Div(Y)_\R \ar[d]\\
\Div(X)_R/N \ar@{^(->}[r]^{f^*} \ar@{->>}[d] & \NS(Y)_\R \\
\NS(X)_W. & }
\end{equation}

Let $D$ be a class of Weil divisor such that $f^*D$ is numerically trivial on $Y$. We will show that $D$ is nef. Since the same reasoning on $-D$ shows that $D$ is anti-nef, we will be done. Let $B$ be any ample $\Q$-Cartier divisor on $Y$. Since $f^*D$ is nef, $f^*D+B$ is ample. By \ref{lm:kleiman4.1}, $D+f_*B$ is nef. Since nefness is a closed property, $D$ is nef.
\end{proof}

\begin{rk} Since $\Div(Y)_\R\rightarrow\NS(Y)_\R$ is surjective, $\Div(X)_\R/N\rightarrow\NS(Y)_\R$ is also surjective, which implies that $\Div(X)_\R/N\cong\NS(Y)_\R$. In particular, there is no natural map $\NS(X)_W\rightarrow\NS(Y)_W$, but the pushforward induces a natural surjective map $\NS(Y)_W\rightarrow\NS(X)_W$.
\end{rk}

\begin{cor}\label{cor:surjectionNSW} Let $f:Y\rightarrow X$ be a small, birational, projective morphism of normal, projective varieties over a field of characteristic $0$, with $X$ having log terminal singularities. The pushforward induces a surjective morphism
$$
f_*:\NS(Y)_W\twoheadrightarrow\NS(X)_W.
$$
\end{cor}

\begin{proof} The diagram in \eqref{eq:f*surj} applies to any $f:Y\rightarrow X$ (small, projective, birational), with $\NS(Y)_W$ instead of $\NS(Y)_{\R}$ if $Y$ is not $\Q$-factorial.
\end{proof}

\begin{rk} The above map is not injective in general. For example, if $D$ is a numerically trivial non-$\Q$-Cartier divisor on $X$, and $f:Y=\Proj_X\Rscr(X,D)\rightarrow X$, then $f^{-1}_*D$ is $f$-ample, and thus not trivial in $\NS(Y)_\R$. However, $f_*(f^{-1}_*D)=D$, which is $0$ in $\NS(X)_W$.
\end{rk}

\begin{df} Let $g:X\rightarrow U$ be a projective morphism of normal, quasi-projective varieties over an algebraically closed field of characteristic $0$, with $X$ having log terminal singularities. The \emph{relative Weil Picard number} is $\rho_W(X/U):=\dim_\R\NS(X/U)_W$. We will write $\rho_W(X)$ for $\rho_W(X/\Spec k)$, and we will simply call it the \emph{Weil Picard number}.
\end{df}

\begin{rk} Proposition \ref{prop:NS<NSW} implies
\begin{equation}
\rho(X)\leq\rho(X)_W.
\end{equation}
\end{rk}

The next result is crucial for the study of positivity on non-$\Q$-factorial singularities.

\begin{thm}\label{thm:nef<=>} Let $X$ be a (normal) projective variety over an algebraically closed field of characteristic $0$ and with log terminal singularities, and let $f:Y\rightarrow X$ be a small, projective, birational map. A Weil $\R$-divisor $D$ on $Y$ is nef if and only if $D$ is $f$-nef and $f_*D$ is nef.
\end{thm}

\begin{proof} Notice that $Y$ has log terminal singularities, by \ref{lm:QCart still lt}. Clearly if $D$ is nef, it is relatively nef. By \ref{cor:pushforwardnef}, $f_*D$ is nef.

Let us prove the converse. Without loss of generality, we can assume that $D$ is a $\Q$-divisor. Let $A$ be an ample $\Q$-Cartier divisor on $Y$, which in particular will be $f$-ample. Since $A$ is ample, $f_*A$ is nef, which implies that $f_*D+f_*A$ is nef as well. Let $H$ be an ample $\Q$-Cartier divisor on $X$. By definition of nefness, for all $m$ sufficiently divisible, $\regsh_X(m(f_*D+f_*A+H))$ is globally generated. Since $D$ is $f$-ample, for $m$ sufficiently divisible $\regsh_Y(m(D+A))$ is relatively globally generated, which in turn implies that $\regsh_Y(m(D+A+f^*H))$ is relatively globally generated as well ($mf^*H$ is $f$-trivial). This means that we have a surjection
$$
f^*f_*\regsh_Y(m(D+A+f^*H))\twoheadrightarrow\regsh_Y(m(D+A+f^*H)).
$$
By the projection formula,
\begin{eqnarray*}
f_*\regsh_Y(m(D+A+f^*H))&\cong& f_*\regsh_Y(m(D+A))\otimes\regsh_X(mH)\cong\\
&\cong&\regsh_X(m(f_*D+f_*A))\otimes\regsh_X(mH)\cong\\
&\cong&\regsh_X(m(f_*D+f_*A+H)),
\end{eqnarray*}
by \ref{lm:pushforwardsheaves}. We have already observed how this sheaf is globally generated, hence so is its pullback $f^*f_*\regsh_Y(m(D+A+f^*H))$. Thus, $\regsh_Y(m(D+A+f^*H))$ is globally generated. Since this is true for any ample $\Q$-Cartier $A$ and $m$ sufficiently divisible, $D+f^*H$ is nef. This is true for every $H$ ample and, since nefness is a closed property, \ref{lm:amplecone=interiornefcone}, $D$ is nef. 
\end{proof}

\begin{qt}\label{qt:nefandrelativenefness} As for \ref{qt:pushforwardnef}, it would be interesting to know if such a result holds without the assumption of $f$ being small.
\end{qt}

\begin{cor}\label{cor:sumofNS} Let $X$ be a (normal) projective variety over an algebraically closed field of characteristic $0$ having log terminal singularities, and let $f:Y\rightarrow X$ be a small, projective, birational map. We have
\begin{equation}\label{eq:sumofNS}
\rho_W(Y)\leq\rho_W(X)+\rho_W(Y/X).
\end{equation}
\end{cor}

\begin{proof} By the previous lemma, a divisor $D$ on $Y$ is numerically trivial if and only if $f_*D$ is numerically trivial on $X$ and $D$ is $f$-numerically trivial. If $f_*:\NS(Y)_W\twoheadrightarrow\NS(X)_W$ is the pushforward and $\pi:\NS(Y)_W\twoheadrightarrow\NS(Y/X)_W$ is the natural quotient map, the previous statement becomes
\begin{equation}\label{eq:capker}
\ker f_*\cap\ker\pi=\{0\},
\end{equation}
which implies \eqref{eq:sumofNS}.
\end{proof}

\vspace{2ex}

We conclude this section by relating the notion of numerically trivial introduced here with the one due to Fulton, \cite[\S19]{fulton}. In \cite[19.1]{fulton}, a $k$-cycle $\alpha$ on a complete scheme $X$ is defined to be \emph{numerically equivalent to zero} if $\int_X P\cap\alpha=0$ for all polynomials $P$ in Chern classes of vector bundles on $X$. The space $N_kX:=Z_kX/\mathrm{Num}_kX$ is the space of all $k$-cycles modulo the group of $k$-cycles numerically equivalent to zero. The following lemma was suggested by M. Fulger.

\begin{lm}\label{lm:mihai} Let $X$ be a normal projective variety over an algebraically closed field and let $D$ be a nef and anti-nef Weil divisor; then $D$ is numerically equivalent to zero, in the sense of \cite{fulton}.
\end{lm}

\begin{proof} Since $D$ is nef, it is pseudo-effective. Let $H$ be a general ample divisor. Since $H$ is general, we can restrict $D$ to $H$ and $D\big|_H$ is still pseudo-effective. Thus, the class intersection $[D]\cap [H]^{n-1}$ is still pseudo-effective. The same is true for $-D$. Hence $[D]\cap [H]^{n-1}=0$. By \cite[3.16]{mihai}, $D$ is numerically equivalent to zero (in the sense of \cite{fulton}).
\end{proof}

\begin{cor} Let $X$ be a normal projective variety over an algebraically closed field of characteristic $0$ having log terminal singularities. There is a natural surjection $\NS(X)_W\twoheadrightarrow N_{n-1}(X)$. In particular, $N_{n-1}(X)$ is finitely dimensional.
\end{cor}

Let us recall that the finite dimensionality of all spaces $N_k(X)$ is proven in \cite[19.1.4]{fulton}.

\section{The Weil nef cone and the cone of Weil curves}\label{sect:cones}

\begin{df}\label{df:Nefcone} Let $X$ be a projective variety with log terminal singularities (over a field of characteristic $0$). We define the \emph{Weil ample cone} (resp. the \emph{Weil nef cone}, resp. the \emph{Weil big cone}, resp. the \emph{Weil pseudo-effective one}) to be the subset of $\NS(X)_W$ generated by the classes of ample (resp. nef, resp. big, resp. pseff) $\Q$-Weil divisors, and it will be denoted by $\Amp(X)_W$ (resp. $\Nef(X)_W$, resp. $\mathrm{Big}(X)_W$, resp. $\mathrm{Pseff}(X)_W$).

Similarly, we define the relative cones $\Nef(X/U)_W$, $\Amp(X/U)_W$.
\end{df}

\begin{rk} By \ref{lm:amplecone=interiornefcone}, amplitude is numerically invariant. Similarly, it is easy to show that bigness and pseudo-effectivity are numerically invariant, so the above definitions make sense.
\end{rk}

These sets are indeed cones and $\Amp(X)_W$ and $\Nef(X)_W$ do not contain lines. Moreover, by \ref{lm:amplecone=interiornefcone},
\begin{eqnarray}\label{eq:nefclosed}
\Nef(X/U)_W=\overline{\Amp(X/U)_W}
\end{eqnarray}
and $\Amp(X/U)_W=\Int(\Nef(X/U)_W)$ (and $\mathrm{Pseff}(X)_W=\overline{\mathrm{Big}(X)_W}$). They are also full dimensional, in the sense that every $\R$-Weil divisor is the difference of two nef/ample $\R$-Weil divisors which is an immediate consequence of \ref{lm:ample+anything}.

\begin{ex}\label{ex:cone1}\label{ex:boundaryandnefness} Let us consider the quadric cone $X=\{xy-zw=0\}$ in $\Pspace^4_\C$. Every Weil divisor is $\Q$-linearly equivalent to a cone $C_D$ over a divisor $D$ in $\Pspace^1_\C\times\Pspace^1_\C$. Thus $\NS(X)_W\cong\R^2$. Moreover, if $D$ is of type $(a,b)$, $C_D$ is nef if and only if $a,b\geq0$ and $C_D$ is ample if and only if $a,b>0$, \cite[2.19]{wStefano2}. Thus the nef cone is $\Nef(X)_W=\{(x,y)\,|\,x,y\geq0\}$ and $\Amp(X)_W=\{(x,y)\,|\,x,y>0\}$. The only Cartier divisors on $X$ are $\Q$-linearly equivalent to the cone over a divisor of type $(1,1)$. Hence, in this case, $\Nef(X)=\{(x,x)\,|\,x\geq0\}$ and $\Amp(X)=\{(x,x)\,|\,x>0\}$. We have the following picture:
$$
\setlength{\unitlength}{0.5cm}
\begin{picture}(10,11)
\put(0,0){\includegraphics[width=5cm,height=5cm]{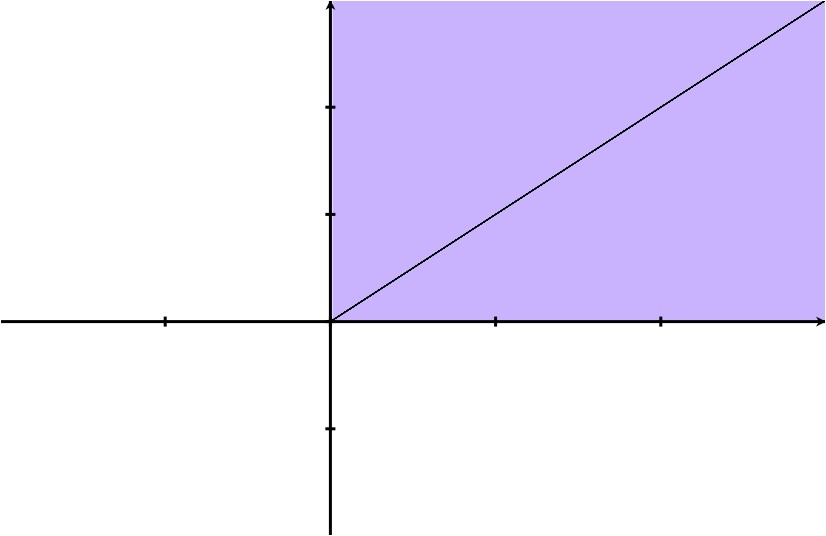}}
\put(8,8){$\displaystyle \Nef(X)$}
\put(8,6){$\Nef(X)_W$}
\put(9,3.3){$C_{(1,0)}$}
\put(2.2,9.3){$C_{(0,1)}$}
\end{picture}
$$

We have $\mathrm{Big}(X)_W=\Amp(X)_W$ and $\mathrm{Pseff}(X)_W=\Nef(X)_W$. In particular, the divisor $C_{(0,-1)}$ is not nef. However, if a divisor $\Delta$ is effective, it must be numerically equivalent to a divisor $C_{(a,b)}$ with $a,b\geq0$. Thus, for any effective Weil divisor $\Delta$, such that $C_{(0,-1)}+\Delta$ is $\Q$-Cartier, $C_{(0,-1)}+\Delta$ is nef (the nearest $\Q$-Cartier divisor possible being $C_{(0,0)}\equiv 0$).

What is happening is the following. It is true that a neighborhood of $D$ in $\NS(X)_W$ avoids the nef cone; however, the $\R$-Cartier divisors are not dense, so we hit the nef cone when we try to modify the Weil divisor with an effective boundary to make it $\Q$-Cartier.
\end{ex}

\begin{rk} The previous example shows that the technique used to obtain contractions in \cite[5.8]{wStefano2} is not satisfactory in general.
\end{rk}

Using duality, we can define the cone of Weil curves.

\begin{df}\label{df:NEcone} Let $X$ be a projective variety  over a field of characteristic $0$ and with log terminal singularities. We define the $N_1(X)_W:=\NS(X)_W^*$, and the \emph{cone of Weil curves} $\NEbar(X)_W:=\Nef(X)_W^*$ (one is dual as vector spaces, the other as cones).

Similarly, if $X\rightarrow U$ is a projective morphism of quasi-projective varieties over an algebraically closed field of characteristic $0$, with $X$ having log terminal singularities, we define $N_1(X/U)_W:=\NS(X/U)_W^*$ and $\NEbar(X/U)_W:=\Nef(X/U)_W^*$.
\end{df}

\begin{rk} The cone $\NEbar(X)_W$ is still closed. Moroever, since these cones are finite dimensional, $\NEbar(X)_W^*\cong\Nef(X)_W$.
\end{rk}

\begin{cor}\label{cor:NEW>NE} Let $X$ be a projective variety with log terminal singularities over an algebraically closed field of characteristic $0$. There is a natural inclusion $\Nef(X)\subseteq\Nef(X)_W$ which induces a natural surjection $\NEbar(X)_W\twoheadrightarrow\NEbar(X)_\R$.
\end{cor}

\begin{proof} The natural inclusion $\Nef(X)\subseteq\Nef(X)_W$ is the restriction of $\NS(X)_\R\subseteq\NS(X)_W$, \ref{prop:NS<NSW}.
\end{proof}

\begin{lm}\label{lm:ample+anything} Let $f:X\rightarrow U$ be a projective morphism of normal, quasi-projective varieties over an algebraically closed field of characteristic $0$, with $X$ having log terminal singularities. For any Weil divisor $D$ and ample $f$-Cartier divisor $A$, there exists $b>0$ such that $bA+D$ is $f$-ample.
\end{lm}

\begin{proof} The statement is local on the base, so we can assume that $U$ is affine. Let $b'$ such that $m(b'A+D)$ is globally generated for some $m\gg0$ (it is enough for $m$ to be divisible by the all the degrees of the generators of $\Rscr(X,D)$). Since all the algebras of local sections are finitely generated, this implies that $b'A+D$ is agg. But then $(b'+1)A+D=(b'A+D)+A$ is ample.
\end{proof}

\begin{lm}\label{lm:pushforwardample} Let $f:Y\rightarrow X$ be a small, birational, projective morphism of normal, projective varieties over an algebraically closed field. For an $\R$-Weil ample divisor $A$ on $Y$, $f_*A$ is ample.
\end{lm}

\begin{proof} As in \ref{lm:kleiman4.1}, it is enough to prove it for $\Q$-Cartier $\Q$-divisors: if $g:Z\rightarrow Y$ is the $\Q$-Cartierization of $A$, $A_Z:=g^{-1}_*A$ is ample, \ref{thm:QCart positivity}, and $(f\circ g)_*(A_Z)=f_*A$.

Let $L$ be a $\Q$-Cartier ample divisor on $X$. Since $A$ is ample, for $0<\e\ll1$, $A-\e f^*L$ is ample. But then $f_*A-\e L=f_*(A-\e f^*L)$ is nef by \ref{lm:kleiman4.1}. Thus $f_*A$ is ample.
\end{proof}

\begin{cor}\label{cor:ample<=>} Let $f:Y\rightarrow X$ be a small, birational, projective morphism of normal, projective varieties over a field of characteristic $0$, with $X$ having log terminal singularities. An $\R$-Weil divisor $A$ on $Y$ is ample if and only if it is $f$-ample and $f_*A$ is ample.
\end{cor}

\begin{proof} If $A$ is ample, it is $f$-ample, and by \ref{lm:pushforwardample} $f_*A$ is ample. Let us prove the converse. Let $L$ be any ample $\Q$-Cartier divisor on $Y$. For positive $b_0$ sufficiently divisible, $b_0A-L$ is $f$-nef since $A$ is $f$-ample. Similarly, for positive sufficiently divisible $b_1$, $b_1f_*A-f_*L$ is nef. Thus, for positive $b$ divisible by $b_0$ and $b_1$, $bA-L$ is $f$-nef and $f_*(bA-L)=bf_*A-f_*L$ is nef. By \ref{thm:nef<=>}, $bA-L$ is nef, which implies that $A$ is ample.
\end{proof}

\vspace{1ex}

If $D$ is an $\R$-Weil divisor on $X$, we define the following subsets (as in the $\Q$-Cartier case):
\begin{itemize}
\item[$\circ$] $D_{>0}:=\{x\in N_1(X)_W\,|\,x(D)>0\}$ (and similarly for $\geq0$, $<0$, $\leq 0$, $=0$);
\item[$\circ$] $\NEbar(X)_{W,D>0}:=\NEbar(X)_W\cap D_{>0}$ (and similarly for $\geq0$, $<0$, $\leq0$, $=0$).
\end{itemize}

\begin{thm}[Kleiman's criterion]\label{thm:kleimanWeil} Let $X$ be a projective variety over an algebraically closed field of characteristic $0$ having log terminal singularities, and let $D$ be an $\R$-Weil divisor on $X$. Then $D$ is ample if and only if $\NEbar(X)_W\setminus\{0\}\subseteq D_{>0}$.
\end{thm}

\begin{proof} The proof proceeds as in the $\Q$-Cartier case in several steps, and is taken from \cite{lazarsfeld}.

\textit{(1)} Since $\NS(X)_W$ is finitely generated, \ref{thm:NSWfg}, $\Nef(X)_W\cong\NE(X)_W^*$. Thus, a divisor $N$ is nef if and only if $x(N)\geq0$ for every $x\in\NEbar(X)_W$.

\textit{(2)} Let us fix an ample divisor $H$. By \ref{lm:amplecone=interiornefcone}, $D$ is ample if and only if there exists $\e>0$ such that 
\begin{eqnarray}
\frac{x(D)}{x(H)}\geq\e,
\end{eqnarray}
for every $x\in\NEbar(X)_W$. Indeed, this equation is true if and only if $D-\e H$ is nef, by (1). If $D$ is ample, then $D-\e H$ is ample for $0<\e\ll1$, \ref{lm:amplecone=interiornefcone}. Conversely, if $D-\e H$ is nef, then $(D-\e H)+\e H=D$ is ample.

\textit{(3)} Fix any norm $||\cdot||$ on $\NS(X)_W^*$. The inclusion $\NEbar(X)_W\setminus\{0\}\subseteq D_{>0}$ is equivalent to the inclusion
\begin{equation}\label{eq:kleiman}
(\NEbar(X)_W\cap S)\subseteq(D_{>0}\cap S),
\end{equation}
where $S=\{||x||=1\}\subset\NS(X)_W^*$. Let us assume \eqref{eq:kleiman}, and show that $D$ is ample. Let $L_D:\NS(X)_W^*\rightarrow\R$ be defined as $L_D(x)=x(D)$. This is a linear functional, which is positive on $\NEbar(X)_W\cap S$. Since $S$ is compact, there is a positive $\e$ such that $x(D)=L_D(x)\geq\e>0$ for all $x\in\NEbar(X)_W\cap S$. Thus $L_D(x)\geq\e||x||$ for all $x\in\NEbar(X)_W$.

\textit{(4)} Let $H_1,\ldots,H_n$ be ample divisors generating $\NS(X)_W$. Since all norms are equivalent on a finite dimensional Banach space, $||\cdot||$ is equivalent to the taxicab norm, $||x||_{\mathrm{taxi}}=\sum|x(H_i)|$. Since $\sum|x(H_i)|\geq|x(H)|$, where $H=\sum H_i$, there exists $\e'>0$ such that $x(D)\geq\e'x(H)$ for all $x\in\NEbar(X)_W$. By (2), this implies that $D$ is ample. 

\textit{(5)} Let $D$ be ample. By duality, for each $x\in\NEbar(X)_W^*\setminus\{0\}$ there exists $H_x\in\NS(X)_W$ such that $x(H_x)>0$. Since $D$ is ample, there exists $b>0$ such that $bD-H_x$ is nef. Then, $0\leq x(bD-H_x)=bx(D)-x(H_x)$, which implies that $x(D)\geq x(H_x)/b>0$.
\end{proof}

As an immediate consequence of \ref{thm:nef<=>}, we can start to understand the cone structure of $\NE(X)_W$.

\begin{lm}\label{lm:NE=sum} Let $X$ be a projective variety over an algebraically closed field of characteristic $0$ having log terminal singularities, and let $f:Y\rightarrow X$ be a small, projective, birational morphism. Then $\NEbar(Y)_W=\NEbar(X/Y)_W+\NEbar(X)_W$. Moreover, $\Int(\NEbar(X)_W)\subseteq\Int(\NEbar(Y)_W)$ and $\Int(\NEbar(X/Y)_W)\subseteq\Int(\NEbar(Y)_W)$.
\end{lm}

\begin{proof} Using cones and the maps $f_*:\NS(Y)_W\twoheadrightarrow\NS(X)_W$ and $\pi:\NS(Y)_W\twoheadrightarrow\NS(Y/X)_W$, we can rephrase \ref{thm:nef<=>} as
\begin{equation}\label{eq:intersectionofnefcones}
\Nef(Y)_W=\pi^{-1}\Nef(Y/X)_W\cap(f_*)^{-1}\Nef(X)_W.
\end{equation}
Moreover, $\ker\pi\cap\ker f_*=\{0\}$. By duality this implies that $N_1(X)_W\hookrightarrow N_1(Y)_W$, $N_1(Y/X)_W\hookrightarrow N_1(Y)_W$ and $N_1(X)_W+N_1(Y/X)_W=N_1(Y)_W$, identifying these spaces with their images. With this identification, \eqref{eq:intersectionofnefcones} becomes
\begin{equation}\label{eq:NE=sum}
\NEbar(Y)_W=\NEbar(X/Y)_W+\NEbar(X)_W.
\end{equation}
The statement on the containments of the interiors is the dual of \ref{cor:ample<=>}.

\end{proof}

\section{The Cone theorem}\label{sect:conethm}

Before proving the Cone theorem, we will prove a global generation theorem, which is a non-$\Q$-Cartier version of the usual basepoint-free theorem, \cite[3.3]{komo}. As in the usual Minimal Model Program, the contraction theorem is a consequence of the basepoint-free theorem. However, global generation and basepoint-freeness do not agree for Weil divisors. For this reason, \ref{thm:globgen} is only a statement about global generation, and not basepoint-freeness. Hence, the linear system associated with $bD$, for $b$ sufficiently divisible, does not induce a morphism.

The theorem below is a stronger version of \cite[5.2]{wStefano2}, which would not be sufficient to prove the Cone theorem \ref{thm:cone}.

\begin{thm}(Global generation)\label{thm:globgen} Let $X$ be a projective variety with log terminal singularities over an algebraically closed field of characteristic $0$. Let $D$ be a nef Weil divisor such that $aD-K_X$ is nef and big for some $a>0$ (resp. for all $a>0$ sufficiently divisible), then $D$ is agg.
\end{thm}

\begin{rk} We do not assume that $K_X$ is $\Q$-Cartier.
\end{rk}

\begin{proof} First, let us assume that $D$ is $\Q$-Cartier. Let $g:Z\rightarrow X$ be the $\Q$-Cartierization of $-K_X$. Since $D$ is $\Q$-Cartier and nef, this will also be the $\Q$-Cartierization of $aD-K_X$ for $a>0$. Then $ag^*D-K_Y=g^{-1}_*(aD-K_X)$ is big and nef (and $\Q$-Cartier). By \ref{lm:QCart still lt}, $Z$ still has $\Q$-Gorenstein log terminal singularities. By the usual basepoint-free theorem, $g^*D$ is agg. Since $D$ is $\Q$-Cartier, by the projection formula $D$ is also agg.

For the general case, let $f:Y\rightarrow X$ be the $\Q$-Cartierization of $D$. By \ref{lm:QCart still lt}, $Y$ still has log terminal singularities. Moreover, $D_Y:=f^{-1}_*D$ is $\Q$-Cartier, nef and $f$-ample.

We will prove that, for $a$ sufficiently divisible, $aD_Y-K_Y$ is nef. Let $A$ be any ample Cartier divisor on $X$. For each $m$ and $a$ sufficiently divisible, $m(aD-K_X+A)$ is globally generated. Moreover, since $D_Y$ is $f$-ample, $m(aD_Y-K_Y)$ is $f$-ample and relatively globally generated for $m$ and $a$ sufficiently divisible. Since $f$ is small,
$$
\varphi:\regsh_Y\cdot\regsh_X(m(aD-K_X+A))\rightarrow\regsh_Y(m(aD_Y-K_Y+f^*A))
$$
is an isomorphism at the level of global sections. Since $aD_Y-K_Y$ is $f$-ample, $\regsh_Y(m(aD_Y-K_Y+g^*A))\otimes\regsh_Y(ng^*A)$ is also globally generated for $n$ sufficiently divisible. But then $\varphi$ must be surjective, and hence and isomorphism. Since the product $\regsh_Y\cdot\regsh_X(m(aD-K_X+A)$ is generated by global sections, by what just observed so is $\regsh_Y(m(aD_Y-K_Y+f^*A))$. This implies that $aD_Y-K_Y+g^*A$ is nef, and since nefness is a closed property on varieties with log terminal singularities, \eqref{eq:nefclosed}, $aD_Y-K_Y$ is nef.

By \cite[3.9]{wStefano2} (the pullback of a big Weil divisor is still big), $aD_Y-K_Y$ is also big. By what we proved at the beginning, $D_Y$ is agg. Since $D_Y$ is $g$-ample, this implies that $D$ is agg.
\end{proof}

\begin{df} If $F$ is a face of $\NEbar(X)_W$, we say that $F$ is \emph{$K_X$-negative} if $x(K_X)<0$ for all $x\in F$.
\end{df}

\begin{df}\label{df:rtlF} Let $F$ be a face of $\NEbar(X)_W$. If there exists a $\Q$-Cartier divisor $D$ on $X$ such that $F=\NEbar(X)_{W,D=0}=\NEbar(X)_W\cap D_{=0}$ (under the natural inclusion $\NS(X)_\R\subseteq\NS(X)_W$, \ref{prop:NS<NSW}), we say that $F$ is a \emph{rational face} of $\NEbar(X)_W$.
\end{df}

In the usual setting, one way of obtaining the Cone theorem is as a consequence of the basepoint-free theorem, \cite[3.3]{komo} and of the rationality theorem, \cite[3.5]{komo}. However, in the non-$\Q$-Gorenstein setting, it is not even clear how a rationality statement could be formulated.

{\renewcommand{\labelenumi}{(\arabic{enumi})}
\begin{thm}[Cone theorem]\label{thm:cone}  Let $X$ be a normal, projective variety over an algebraically closed field of characteristic $0$, with log terminal singularities.
\begin{enumerate}
\item There are countably many $C_j\in N_1(X)_W$ such that
$$
\NEbar(X)_W=\NEbar(X)_{W,K_X\geq0}+\sum\R_{\geq0}\cdot C_j,
$$
and they do not accumulate on the half-space $K_{X<0}$.

\item\label{item:contraction} (Contraction theorem) Let $F\subset\NEbar(X)_W$ be a $K_X$-negative extremal face. There exists a diagram
\begin{equation}
\xymatrix{& \Xtilde\ar[dr]^{\varphitilde_F} \ar[dl]_f & \\ X \ar@{-->}[rr]_{\varphi_F} & & Y,}
\end{equation}
where 
\begin{enumerate}
\item $f$ is small, projective, birational, and $\rho(\Xtilde)\leq\rho(X)+1$;
\item under the inclusion $\NEbar(X)_W\hookrightarrow\NEbar(\Xtilde)_W$, $F$ is a rational $K_{\Xtilde}$-negative extremal face and
%
$\varphitilde_F$ is a contraction of the face $F$;
\item $\Xtilde$ has log terminal singularities and $\varphitilde_{F *}\regsh_{\Xtilde}=\regsh_Y$.
\end{enumerate}

\end{enumerate}
\end{thm}}

{\renewcommand{\labelenumi}{(\arabic{enumi})}
\begin{proof}
\begin{enumerate}
\item Let $f:Z\rightarrow X$ be a $\Q$-factorialization of $X$ such that $K_Z$ is $f$-nef; such map exists by \cite[109]{kollarexercises}. We know that $Z$ still has log terminal singularities, \ref{lm:QCart still lt}. We have that
$$
\NEbar(Z)=\NEbar(Z/X)+\NEbar(X)_W,
$$
by \eqref{eq:NE=sum} and since $\NEbar(Z)_W=\NEbar(Z)$ and $\NEbar(Z/X)_W=\NEbar(Z/W)$. By construction, $\NEbar(Z/X)_{K_Z\geq0}=\NEbar(Z/X)$, or equivalently $\NEbar(Z/X)_{K_Z<0}=\emptyset$. Thus
$$
\NEbar(Z)_{K_Z<0}=\NEbar(X)_{W,K_X<0}.
$$
By the usual Cone theorem, \cite[3.3]{komo},
$$
\NEbar(Z)_{K_Z<0}=\sum_{countable} \R_{\geq0}\cdot C_j
$$
and they do not accumulate on the half-space $K_{Z<0}$, which concludes the proof.
\item The first part of this proof follows \cite[\S3.3, Steps 6 and 7]{komo}. Let $\langle F\rangle\subset\NS(X)_W^*$ be the linear span of $F$. Let $H$ be an ample divisor and $\e>0$ such that $K_X+\e H$ is negative on $F$. Notice that, since $F$ is extremal, $\langle F\rangle\cap\NEbar(X)_W=F$. Let
$$
W_F:=\NEbar(X)_{W,K_X+\e H\geq0}+\sum_{\dim R=1, R\nsubset F} R,
$$
where $R$ are extremal rays of $\NEbar(X)_W$. Then $W_F$ is a closed cone which intersect $\langle F\rangle$ only at the origin and $\NEbar(X)_W=W_F+F$. Thus, there is an $\R$-Weil divisor $G$ such that $\langle F\rangle\subseteq G_{=0}$, but $G_{=0}\cap W_F=\{0\}$ (where $G_{=0}:=G_{\geq0}\cap G_{\leq0}$). Moreover, such $\R$-Weil divisor $D$ can be chosen to be positive on $W_F$. Notice that such $G$ will necessarily be nef. The set of such divisors is open and non-empty, thus there must be one defined over $\Q$. By taking suitable multiples, we can assume that $D$ is a nef Weil divisor which is positive on $W_F$ and zero exactly on $F$. Since $-K_X$ is positive on $F$, while $D$ is positive on $W_F$ and zero on $F$, for $b>0$, $bD-K_X$ is positive on $\NEbar(X)_W$. By \ref{thm:kleimanWeil}, $bD-K_X$ is ample. By \ref{thm:globgen}, $D$ is agg. 

Let $\Xtilde=\Proj_X\Rscr(X,D)$, which has log terminal singularities by \ref{lm:QCart still lt}. Then $\Dtilde:=f^{-1}_*D$ is $\Q$-Cartier. Let $m$ be sufficiently divisible. Since $\regsh_X(mD)$ is globally generated, so is $f^*\regsh_X(mD)$ and thus so is $\regsh_{\Xtilde}\cdot\regsh_X(mD)$, since $\regsh_{\Xtilde}\cdot\regsh_X(mD)$ is the quotient of $f^*\regsh_X(mD)$ by its torsion. Since we have an inclusion $\regsh_{\Xtilde}\cdot\regsh_X(m D)\hookrightarrow\regsh_{\Xtilde}(m\Dtilde)$ which is an isomorphism at the level of global sections, $\regsh_{\Xtilde}(m\Dtilde)$ is globally generated. Hence a high enough power of $\Dtilde$ defines a morphism whose Stein factorization is $\varphitilde_D:\Xtilde\rightarrow Y$, and we set $\varphitilde_F:=\varphitilde_D$.

With slight abuse of notation, we will identify $\NEbar(X)_W$ and $\NEbar(\Xtilde/X)_W$ as subspaces of $\NEbar(\Xtilde)_W$. With this identification, if $x\in\NEbar(X)_W$, $x(D)=x(\Dtilde)$ and $x(K_X)=x(K_{\Xtilde})$. We immediately have that $F$ is still $K_{\Xtilde}$-negative. Moreover, the divisor $\Dtilde$ is $f$-ample, thus by the relative version of Kleiman's criterion, \ref{thm:kleimanWeil,rel}, $\NEbar(\Xtilde/X)_W\setminus\{0\}\subseteq\Dtilde_{>0}$. Moreover, $\NEbar(X)_W\subseteq\Dtilde_{\geq0}$ and $F=\NEbar(X)\cap D_{=0}$. Thus,
\begin{eqnarray*}
\NEbar(\Xtilde)_W\cap\Dtilde_{=0}&=&\big(\NEbar(X)_W+\NEbar(\Xtilde/X)_W\big)\cap\Dtilde_{=0}=\\
&=&\NEbar(X)_W\cap\Dtilde_{=0}=\NEbar(X)_W\cap D_{=0}=F.
\end{eqnarray*}
This implies that $F$ is a rational extremal face.

As in the first part of the proof (but now for $\Xtilde$ and $F$), $F$ is a $K_{\Xtilde}$-negative extremal face, $\Dtilde$ is an asymptotically globally generated nef divisor such that $\NEbar(\Xtilde)_W\cap\Dtilde_{\geq0}=F$, so as in \cite[\S3.3, Steps 6 and 7]{komo}, it uniquely defines a morphism whose Stein factorization $\varphitilde_F:\Xtilde\rightarrow Y$, with $\varphitilde_{F *}\regsh_{\Xtilde}=\regsh_Y$, is the desired contraction.
\end{enumerate}
\end{proof}}

\begin{rk}[The map $f:\Xtilde\rightarrow X$ and flips] It is reasonable to ask whether the map $f:\Xtilde\rightarrow X$ of the theorem is a flip. More precisely, we might have performed a flipping contraction in the previous step, and $f:\Xtilde\rightarrow X$ might be the flip of such contraction:
$$
\xymatrix{X_{i-1}\ar@{-->}[rr] \ar[dr] & & \Xtilde=X_{i-1}^+ \ar[dl]^f\\
& X_i. &}
$$
This is certainly possible. However, this is not the only possibility. If the face we are contracting is already rational $\Xtilde=X$ and $f=id$, in which case this is not a flip. Even when $\Xtilde=\Proj_X\Rscr(X,D)$ for some Weil divisor $D$, this would be a flip only if $D$ and $K_X$ differed by a line bundle, which might not happen. The point of this program is not that we will never do flips, but that we are not forced to perform them (as in the usual MMP).
\end{rk}

As a corollary, we obtain the usual description of the effective cones for Fano varieties. In particular, this result will apply to the generic fiber of a contraction of the Cone theorem (by \ref{lm:Moricontraction} and the relative version of Kleiman's criterion \ref{thm:kleimanWeil,rel}).

\begin{cor} Let $X$ be a projective variety over a field of characteristic $0$ with log terminal singularities and with $-K_X$ ample. Then $\NEbar(X)_W$ is polyhedral (generated by finitely many rays).
\end{cor}

\begin{ex}\label{ex:cone2} Let $X$ be the quadric cone $X=\{xy-zw=0\}$ in $\Pspace^4_\C$. Notice that $X$ is Fano with $\rho(X)=1$. Following the computations of \ref{ex:cone1}, we have that $N_1(X)_W=\R\gamma_{(1,0)}\oplus\R\gamma_{(0,1)}$, where $\gamma_{(1,0)}(C_{(a,b)})=a$ and $\gamma_{(0,1)}(C_{(a,b)})=b$. With this notation, $\NEbar(X)_W=\{a\gamma_{(1,0)}+b\gamma_{(0,1)}\,|\,a,b\geq0\}$.
$$
\setlength{\unitlength}{0.5cm}
\begin{picture}(10,11)
\put(0,0){\includegraphics[width=5cm,height=5cm]{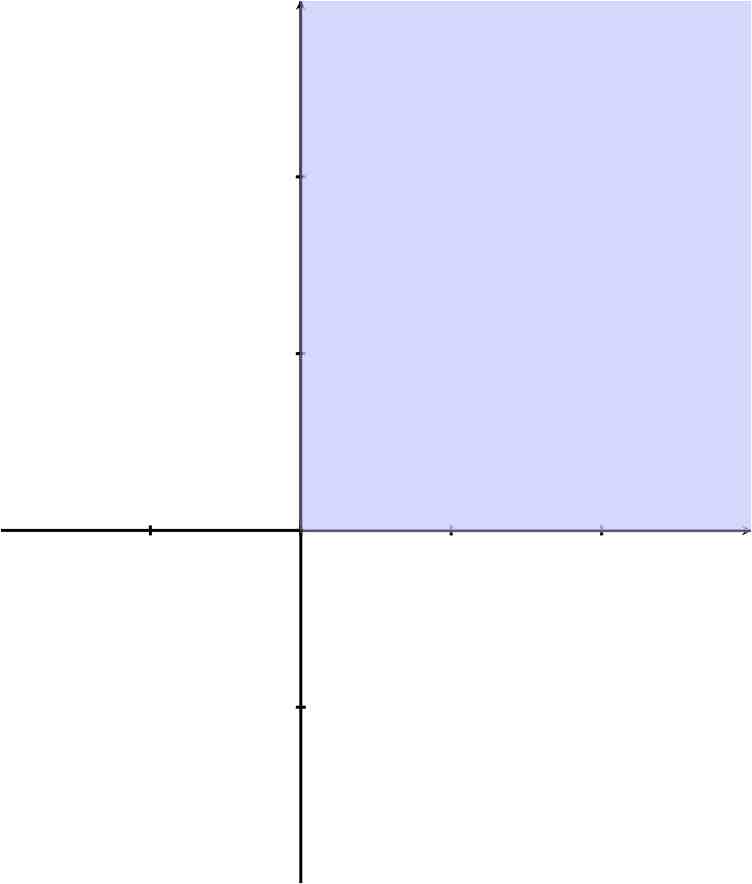}}
\put(8,8){$\NEbar(X)_W$}
\put(9,3.3){$\gamma_{(0,1)}$}
\put(2.2,9.3){$\gamma_{(1,0)}$}
\end{picture}
$$

The cone $\NEbar(X)_W$ has four $K_X$-negative extremal faces:
\begin{enumerate}
\item $F=0$;
\item $F=\{a\gamma_{(1,0)}\,|\,a\geq0\}$;
\item $F=\{b\gamma_{(0,1)}\,|\,b\geq0\}$;
\item $F=\NEbar(X)_W$.
\end{enumerate}
\end{ex}

The first and the last case are rational faces, and in both cases $\Xtilde=X$ with the notation of the theorem. The contraction of the former face is not a contraction at all, but an embedding of $X$ in a projective space, while the contraction of the latter gives a map $X\rightarrow\Spec\C$. These maps are the ones from the regular MMP.

We have two new maps, corresponding to the two faces in (b) and (c). Let us consider the case in (b) (the one in (c) being similar). The divisor $D$ of the theorem describing $F$ is $D:=C_{(0,1)}$, which is agg but not $\Q$-Cartier. We set $\Xtilde=\Proj_X\Rscr(X,D)$, which in this case is actually a resolution, and a $\Pspace^1_\C$-bundle over $\Pspace^2_\C$. The contraction of the theorem is the map
$$
\xymatrix{& \Xtilde\ar[dr]^{\varphitilde_F} \ar[dl]_f & \\ X \ar@{-->}[rr]_{\varphi_F} & & \Pspace^2_\C,}
$$
where the generic fiber of $\varphitilde_F$ is $\Pspace^1_\C$. The geometric picture is the following: $X$ is a cone over a linear embedding of $\Pspace^1_\C\times\Pspace^1_\C$, and $\gamma_{(1,0)}$ corresponds to one of the $\Pspace^1_\C$; when we contract it, we are left with a cone over a linear embedding of $\Pspace^1_\C$, namely $\Pspace^2_\C$.

\section{The program}\label{sect:program}

We start by presenting a Minimal Model Program without flips. We point out that we can obtain a $\Q$-factorial model. 
{\renewcommand{\labelenumi}{\textit{Step \arabic{enumi}}}
\begin{enumerate}
\setcounter{enumi}{-1}
\item Let $X$ be a projective variety with log terminal singularities (over an algebraically closed field of characteristic $0$).
\item \begin{enumerate}
\item If $K_X$ is nef, let $\Xhat\rightarrow X$ be a $\Q$-factorialization of $X$ -- that is, a small projective birational morphism with $\Xhat$ $\Q$-factorial -- with $K_{\Xhat}$ relatively nef. Then, $\Xhat$ is $\Q$-factorial, with log terminal singularities and $K_{\Xhat}$ is nef. The variety $\Xhat$ is a minimal model for $X$.
\item Otherwise, let 
$$
\xymatrix{& \Xtilde\ar[dr]^{\varphitilde} \ar[dl]_f & \\ X \ar@{-->}[rr]_{\varphi} & & Y,}
$$
be a contraction of a $K_X$-negative extremal face, given by the Cone theorem \ref{thm:cone}.
\end{enumerate}
\item There are two possibilities for the contraction above.
\begin{enumerate}
\item If $\dim Y<\dim X$, the generic fiber of $\varphitilde$ is a log Fano variety. We set $X_0:=Y$ and we can restart the program in smaller dimension.
\item Otherwise, we set $X_{i+1}:=Y$, and we restart the algorithm.
\end{enumerate}
\end{enumerate}}

There are several statements that are more or less implicit in the stated version of the program and that need to be proven.

\begin{lm} Let $X$ be a projective variety over an algebraically closed field of characteristic $0$ with log terminal singularities and with $K_X$ nef, and let $f:\Xhat\rightarrow X$ be a $\Q$-factorialization of $X$ with $K_{\Xhat}$ $f$-nef. Then $\Xhat$ has log terminal singularities and $K_{\Xhat}$ is nef. 
\end{lm}

\begin{proof} Notice that such $\Q$-factorialization exists because $X$ has log terminal singularities, \cite[109]{kollarexercises}. Since $f$ is small, $\Xhat$ has log terminal singularities by \ref{lm:QCart still lt}. Since $K_{\Xhat}$ is $f$-nef and $f_*K_{\Xhat}=K_X$ is nef, $K_{\Xhat}$ is nef by \ref{thm:nef<=>}.
\end{proof}

\begin{df} A projective variety $X$ is called \emph{log Fano} if there exists a divisor $\Delta$ on $X$ such that $(X,\Delta)$ is a klt pair and $-(K_X+\Delta)$ is ample.
\end{df}

\begin{lm}\label{lm:Moricontraction} Let $X$ be a projective variety over an algebraically closed field of characteristic $0$ with log terminal singularities, and let $\varphi:X\dashrightarrow Y$ be a contraction of a $K_X$-negative extremal face, with $f:\Xtilde\rightarrow X$ and $\varphitilde:\Xtilde\rightarrow X$ as in the Cone theorem. The varieties $\Xtilde$ and $Y$ have log terminal singularities. Moreover, if $\dim Y<\dim X$, the generic fiber of $\varphitilde$ is a log Fano.
\end{lm}

\begin{proof} The variety $\Xtilde$ has log terminal singularities by \ref{lm:QCart still lt}. By the relative version of Kleiman's criterion for ampleness, \ref{thm:kleimanWeil,rel} , $-K_{\Xtilde}$ is $\varphitilde$-ample; thus $Y$ has log terminal singularities by \ref{prop:contr still lt,2}.

If $\dim Y<\dim X=\dim\Xtilde$, since both $\Xtilde$ and $Y$ have log terminal singularities, and we are over an algebraically closed field of characteristic $0$, the generic fiber of $\varphitilde$ has log terminal singularities. By \ref{lm:logFano}, the generic fiber of $\varphitilde$ is a log Fano.
\end{proof}

\begin{rk}[Outputs]\label{rk:models} This program will produce different outputs than the usual MMP. More precisely, if $(X,\Delta)$ is a klt pair with $K_X+\Delta$ pseudoeffective and $\Delta$ big, the MMP with scaling produces a minimal model $(\Xtilde,\tilde{\Delta})$ with $(\Xtilde,\tilde{\Delta})$ a klt pair with $K_{\Xtilde}+\tilde{\Delta}$ nef. With the same input, and assuming termination of my program, it is not clear that I would obtain a minimal model $\Xhat$. For example, it might happen that $K_X+\Delta$ is pseudo effective but $K_X$ is not, so my program might terminate with a Fano contraction. Even when my program would produce a minimal model $\Xhat$ such model would be with log terminal singularities (in the usual sense) and such that $K_{\Xhat}$ is nef. So this would not necessarily be a model for the pair $(X,\Delta)$. It is important to compare the outputs when $\Delta=0$, which is the only case where we have any hope of obtaining models that we can relate. However, it is not immediately clear what the relation between models would be (alternatively, my program can be extended to pairs, and it is worth exploring what happens for pairs).
\end{rk}

\section{The relative statements}\label{sect:rel}

Almost every result of this paper holds in the relative setting, with the obvious modifications. We will simply list the statements where the modification of the proof is clear.

\begin{prop}[cfr. \ref{prop:NS<NSW} and \ref{cor:NEW>NE}] Let $X\rightarrow U$ be a projective morphism of normal, quasi-projective varieties over an algebraically closed field of characteristic $0$, with $X$ having log terminal singularities. There is a natural inclusion $\NS(X/U)_\R\subseteq\NS(X/U)_W$ that restricts to an inclusion $\Nef(X/U)\subseteq\Nef(X/U)_W$, and that induces surjections $N_1(X/U)_W\twoheadrightarrow N_1(X/U)_\R$ and $\NEbar(X/U)_W\twoheadrightarrow\NEbar(X/U)_\R$.
\end{prop}

\begin{prop}[cfr. \ref{thm:nef<=>} and \ref{cor:ample<=>}]\label{prop:nef<=>,rel} Let $g:X\rightarrow U$ projective morphism of normal, quasi-projective varieties over an algebraically closed field of characteristic $0$, with $X$ having log terminal singularities, and let $f:Y\rightarrow X$ be a small, projective, birational map, and let $h=g\circ f$. A Weil $\R$-divisor $D$ on $Y$ is $h$-nef (resp. $h$-ample) if and only if $D$ is $f$-nef (resp. $f$-ample) and $f_*D$ is $g$-nef (resp. $g$-ample).
\end{prop}

\begin{proof} Clearly, if $D$ is $h$-nef it is $f$-nef. The proof that if $D$ is $f$-nef and $f_*D$ is $g$-nef, then $D$ is $h$-nef is exactly as in \ref{thm:nef<=>}. To see that if $D$ is $h$-nef, then $f_*D$ is $g$-nef, we reduce to the case where $D$ is $\Q$-Cartier and ample, as in \ref{cor:pushforwardnef}. Then, the proof proceeds like the one of \ref{lm:kleiman4.1}, where the usual Castelnuovo-Mumford regularity is substituted by the relative regularity, and  Leray's spectral sequence by Grothendieck's spectral sequence.

The statement for ampleness has the same proof as in the non-relative setting.
\end{proof}

\begin{cor}[cfr. \ref{cor:surjectionNSW}]\label{cor:surjectionNSW,rel} If $X\rightarrow U$ is a projective morphism of normal, quasi-projective varieties over an algebraically closed field of characteristic $0$, with $X$ having log terminal singularities, and $f:Y\rightarrow X$ is a small projective morphism over $U$ with $Y$ normal quasi-projective, as above the pushforward induces a surjection
$$
f_*:\NS(Y/U)_W\twoheadrightarrow\NS(X/U)_W.
$$
By duality, there is a natural inclusion
$$
(f_*)^*:N_1(X/U)_W\hookrightarrow N_1(Y/U)_W.
$$
\end{cor}

\begin{cor}[cfr. \ref{lm:NE=sum}]\label{cor:NE=sum} If $X\rightarrow U$ is a projective morphism of normal quasi-projective varieties over an algebraically closed field of characteristic $0$, with $X$ having log terminal singularities, and $f:Y\rightarrow X$ is a small projective morphism over $U$ with $Y$ normal quasi-projective, then
$$
\Nef(Y/U)_W=\Nef(X/U)_W\cap\Nef(Y/X)_W
$$
and
$$
\NEbar(Y/U)_W=\NEbar(X/U)_W+\NEbar(Y/X)_W.
$$

Moreover, we have $\Int(\NEbar(X/U)_W)\subseteq\Int(\NEbar(Y/U)_W)$ and $\Int(\NEbar(Y/X)_W)\subseteq\Int(\NEbar(Y/U)_W)$.
\end{cor}

\begin{proof} The proof of this statement is pure convex geometry, so it translates without modifications in the relative setting.
\end{proof}

\begin{thm}[Relative version of Kleiman's criterion, cfr. \ref{thm:kleimanWeil}]\label{thm:kleimanWeil,rel} Let $f:Y\rightarrow X$ be a projective  morphism between normal quasi-projective varieties over an algebraically closed field of characteristic $0$, with $Y$ with log terminal singularities. Let $\pi:\NS(Y)_W\rightarrow\NS(Y/X)_W$. Then $D$ on $Y$ is $f$-ample if and only if $\NEbar(Y/X)_W\setminus\{0\}\subseteq\pi(D_{>0})$.
\end{thm}

\begin{proof} Also in this case the proof is solely convex geometry.
\end{proof}

\begin{thm}[cfr. \ref{thm:globgen}]\label{thm:globgen,del} Let $f:X\rightarrow U$ be a projective morphism of normal quasi-projective varieties, with $X$ log terminal singularities, over an algebraically closed field of characteristic $0$. Let $D$ be an $f$-nef Weil divisor such that $aD-K_X$ is $f$-nef and $f$-big for some $a>0$ (resp. for all $a>0$ sufficiently divisible). Then $D$ is $f$-agg.
\end{thm}

\begin{proof} The proof is essentially the same.
\end{proof}

{\renewcommand{\labelenumi}{(\arabic{enumi})}
\begin{thm}[Relative Cone theorem, cfr. \ref{thm:cone}]  Let $X\rightarrow U$ be a projective morphism of normal quasi-projective varieties over an algebraically closed field of characteristic $0$, and let $X$ have log terminal singularities. Then
\begin{enumerate}
\item There are (countably many) $C_j\in N_1(X/U)_W$ such that
$$
\NEbar(X/U)_W=\NEbar(X/U)_{W,K_X\geq0}+\sum\R_{\geq0}\cdot C_j,
$$
and they do not accumulate on the half-space $K_{X<0}$.

\item (Relative contraction theorem) Let $F\subset\NEbar(X/U)_W$ be a $K_X$-negative extremal face. There exists a diagram
\begin{equation}
\xymatrix{& \Xtilde\ar[dr]^{\varphitilde_F} \ar[dl]_f & \\ X \ar@{-->}[rr]_{\varphi_F} \ar[dr] & & Y \ar[dl] \\ & U, &}
\end{equation}
where 
\begin{enumerate}
\item $f$ is small, projective, birational and $\rho(\Xtilde)\leq\rho(X)+1$;
\item under the inclusion $\NEbar(X/U)_W\hookrightarrow\NEbar(\Xtilde/U)_W$, $F$ is a rational $K_{\Xtilde}$-negative extremal face and $\varphitilde_F$ is a contraction of the face $F$;
\item $\Xtilde$ has log terminal singularities and $\varphitilde_{F *}\regsh_{\Xtilde}=\regsh_Y$.
\end{enumerate}\
\end{enumerate}
\end{thm}}

\begin{proof} In this case as well, the proof is essentially the same.
\end{proof}


\begin{thebibliography}{}

\bibitem[BCHM10]{BCHM}
C.~Birkar, P.~Cascini, C.~Hacon, and J.~McKernan, \emph{Existence of minimal
  models for varieties of log general type}, J. Amer. Math. Soc. \textbf{23}
  (2010), 405--468.

\bibitem[BdFFU13]{BdFFU}
S.~Boucksom, T.~de~Fernex, C.~Favre, and S.~Urbinati, \emph{Valuation spaces
  and multiplier ideals on singular varieties}, \url{arXiv:1307.0227}, 2013.

\bibitem[Chi13]{ltII}
A.~Chiecchio, \emph{Some properties and examples of log terminal${}^+$
  singularities}, \url{arXiv:1301.4677}, 2013.

\bibitem[CU12]{wStefano}
A.~Chiecchio and S.~Urbinati, \emph{Log terminal singularities},
  \url{arXiv:1211.6235}, 2012.

\bibitem[CU13]{wStefano2}
\bysame, \emph{Ample {W}eil divisors}, \url{arXiv:1310.5961}, 2013.

\bibitem[dFH09]{dFH}
T.~de~Fernex and C.~Hacon, \emph{Singularities on normal varieties}, Compositio
  Mathematica \textbf{145} (2009), 393--414.

\bibitem[FL14]{mihai}
M.~Fulger and B.~Lehmann, \emph{Positive cones of dual cycle classes},
  \url{arXiv:1408.5154}, 2014.

\bibitem[Fuj99]{fujino_kawamata}
O.~Fujino, \emph{Applications of {K}awamata's positivity theorem}, Proceedings
  of the Japan Academy, Series A, Mathematical Sciences \textbf{75} (1999),
  no.~6, 75--79.

\bibitem[Ful98]{fulton}
W.~Fulton, \emph{Intersection theory}, vol.~93, Springer Berlin, 1998.

\bibitem[Gro60]{ega}
A.~Grothendieck, \emph{Elements de g\'eom\'etrie alg\'ebrique {I}}, 1960.

\bibitem[Har77]{har}
R.~Hartshorne, \emph{Algebraic geometry}, vol.~52, Springer, 1977.

\bibitem[Kle66]{kle}
S.~Kleiman, \emph{Toward a numerical theory of ampleness}, Annals of
  Mathematics (1966), 293--344.

\bibitem[KM98]{komo}
J.~Koll\'ar and S.~Mori, \emph{Birational geometry of algebraic varieties},
  vol. 134, Cambridge Tracts in Mathematics, 1998.

\bibitem[Kol86]{kolhighI}
J.~Koll\'ar, \emph{Higher direct images of dualizing sheaves {I}}, Annals of
  Mathematics (1986), 11--42.

\bibitem[Kol08]{kollarexercises}
\bysame, \emph{Exercises in the birational geometry of algebraic varieties},
  \url{arXiv:0809.2579}, 2008.

\bibitem[Laz04]{lazarsfeld}
R.~Lazarsfeld, \emph{Positivity in algebraic geometry {I}, {II}}, vol. 48, 49,
  Springer, 2004.

\bibitem[Urb11]{stefano2}
S.~Urbinati, \emph{Divisorial models of normal varieties},
  \url{arXiv:1211.1692}, 2011.

\end{thebibliography}
\providecommand{\href}[2]{#2}

\end{document}